\documentclass[reqno,10pt,centertags]{amsart}
\usepackage{amsmath,amsthm,amscd,amssymb,latexsym,upref,enumerate}
\usepackage{hyperref} 

\newcommand*{\mailto}[1]{\href{mailto:#1}{\nolinkurl{#1}}}
\newcommand{\arxiv}[1]{\href{http://arxiv.org/abs/#1}{arXiv:#1}}

\newcommand{\bbC}{{\mathbb{C}}}

\newcommand{\bbN}{{\mathbb{N}}}

\newcommand{\bbR}{{\mathbb{R}}}
\newcommand{\bbS}{{\mathbb{S}}}

\newcommand{\cB}{{\mathcal B}}

\newcommand{\cH}{{\mathcal H}}

\newcommand{\cK}{{\mathcal K}}

\newcommand{\cS}{{\mathcal S}}

\newcommand{\cW}{{\mathcal W}}

\newcommand{\gq}{\mathfrak q} 

\newcommand{\B}{\vec{B}}

\newcommand{\e}{{\varepsilon}}

\newcommand{\Sect}{\text{\rm Sect}}

\DeclareMathOperator{\ran}{ran}
\DeclareMathOperator{\dom}{dom}

\DeclareMathOperator{\tr}{tr}

\renewcommand{\Re}{\text{\rm Re}}
\renewcommand{\Im}{\text{\rm Im}}
\renewcommand{\ln}{\text{\rm ln}}

\newcommand{\loc}{\text{\rm{loc}}}

\newcommand{\locunif}{\operatorname{loc \, unif}}

\newcommand{\beq}{\begin{equation}}
\newcommand{\enq}{\end{equation}}

\newcommand{\p}{\prime}

\newcommand{\no}{\notag}
\newcommand{\lb}{\label}
\newcommand{\f}{\frac}

\newcommand{\ol}{\overline}

\newcommand{\bi}{\bibitem}

\newcommand{\dott}{\,\cdot\,}

\let\geq\geqslant
\let\leq\leqslant

\newcommand{\dv}{\operatorname{div}}

\makeatletter
\def\theequation{\@arabic\c@equation}

\allowdisplaybreaks 
\numberwithin{equation}{section}

\newtheorem{theorem}{Theorem}[section]
\newtheorem{proposition}[theorem]{Proposition}
\newtheorem{lemma}[theorem]{Lemma}
\newtheorem{corollary}[theorem]{Corollary}
\newtheorem{hypothesis}[theorem]{Hypothesis}
\newtheorem{example}[theorem]{Example}

\theoremstyle{remark}
\newtheorem{remark}[theorem]{Remark}

\begin{document}

\title[Square Root Domains]{On Square Root Domains for Non-Self-Adjoint Sturm--Liouville Operators}

\author[F.\ Gesztesy]{Fritz Gesztesy}
\address{Department of Mathematics,
University of Missouri, Columbia, MO 65211, USA}
\email{\mailto{gesztesyf@missouri.edu}}
\urladdr{\url{http://www.math.missouri.edu/personnel/faculty/gesztesyf.html}}

\author[S.\ Hofmann]{Steve Hofmann}
\address{Department of Mathematics,
University of Missouri, Columbia, MO 65211, USA}
\email{\mailto{hofmanns@missouri.edu}}
\urladdr{\url{http://www.math.missouri.edu/$\sim$hofmann/}}

\author[R.\ Nichols]{Roger Nichols}
\address{Mathematics Department, The University of Tennessee at Chattanooga, 
415 EMCS Building, Dept. 6956, 615 McCallie Ave, Chattanooga, TN 37403, USA}
\email{\mailto{Roger-Nichols@utc.edu}}
\urladdr{\url{https://www.utc.edu/faculty/roger-nichols/}}

\dedicatory{Dedicated with great pleasure to Fedor S. Rofe-Beketov on the occasion of 
his 80th birthday.}
\thanks{Appeared in Methods Funct.~Anal.~Topology {\bf 19}, No.~3, 227--259 (2013).}
\date{\today}

\begin{abstract} 
We determine square root domains for non-self-adjoint Sturm--Liouville operators of the type 
$$
L_{p,q,r,s} = - \frac{d}{dx}p\frac{d}{dx}+r\frac{d}{dx}-\frac{d}{dx}s+q  
$$
in $L^2((c,d);dx)$, where either $(c,d)$ coincides with the real line $\bbR$, the half-line $(a,\infty)$, 
$a \in \bbR$, or with the bounded interval $(a,b) \subset \bbR$, under very general conditions on 
the coefficients $q, r, s$. We treat Dirichlet and Neumann boundary conditions at $a$ in the half-line 
case, and Dirichlet and/or Neumann boundary conditions at $a,b$ in the final interval context. (In the 
particular case $p=1$ a.e.\ on $(a,b)$, we treat all separated boundary conditions at $a, b$.) 
\end{abstract}

\maketitle

\section{Introduction}  \lb{s1}

This paper is dedicated to Fedor Rofe-Beketov, an eminent mathematician, at the occasion of his 
80th birthday. Fedor's influence on modern spectral theory of differential operators has been profound 
as is evidenced, for instance, by the monograph \cite{RBK05}. In particular, his work on periodic 
one-dimensional Schr\"odinger operators in the non-self-adjoint as well as self-adjoint context, and the associated perturbation theory in the latter case, is legendary (cf.\ his 
tremendously influential papers \cite{Ro63}, \cite{RB64}, \cite{RB73}) and we hope this contribution to 
Sturm--Liouville theory will create some joy for him. 

The principal aim of this paper is to apply the abstract approach to the problem of stability 
of square root domains of a class of non-self-adjoint operators with respect to additive 
perturbations developed in \cite{GHN13}, to the concrete case of Sturm--Liouville operators of 
the form 
\begin{equation}\lb{1.1}
L_{p,q,r,s} = - \frac{d}{dx}p\frac{d}{dx}+r\frac{d}{dx}-\frac{d}{dx}s+q  
\end{equation}
in $L^2((c,d);dx)$, with $(c,d) = \bbR$, or $(c,d) = (a,\infty)$, or $(c,d)=(a,b)$ with 
$ - \infty < a < b < \infty$. For some constants 
$0<\lambda\leq \Lambda < \infty$ we assume 
\begin{equation} 
\lambda<\Re(p(x)), \quad |p(x)|\leq \Lambda \, \text{ a.e.\ on $(c,d)$},    \lb{1.2} 
\end{equation}  
and either 
\begin{equation} 
q\in L_{\locunif}^1(\bbR;dx), \quad r,s\in L_{\locunif}^2(\bbR;dx),    \lb{1.3} 
\end{equation} 
or (cf.\ \eqref{4.2}) 
\begin{equation} 
q\in L_{\locunif}^1([a,\infty);dx), \quad r,s\in L_{\locunif}^2([a,\infty);dx),    \lb{1.3a} 
\end{equation} 
assuming Dirichlet or Neumann boundary conditions at $a$, or 
\begin{equation} 
q\in L^1((a,b);dx), \quad r,s\in L^2((a,b);dx).      \lb{1.4}
\end{equation}  
In the latter case we impose Dirichlet and/or Neumann boundary conditions at $a,b$, and in the particular 
case $p=1$ a.e.\ on $(a,b)$, we also treat all separated boundary conditions given by 
\begin{equation}
g(a)\cos(\theta_a) + g^{[1]}(a)\sin(\theta_a) = 0,  \quad 
g(b)\cos(\theta_b) - g^{[1]}(b)\sin(\theta_b) = 0, \quad \theta_a,\theta_b\in \bbS_{\pi},    \lb{1.5} 
\end{equation}
where $y^{[1]}(x)=p(x)y'(x)$ for a.e.\ $x\in(a,b)$ denotes the first quasi-derivative of $y \in AC([a,b])$ and 
$\bbS_{\pi}$ abbreviates the strip 
\begin{equation}
\bbS_{\pi}=\{z\in \bbC\, |\, 0\leq \Re(z)<\pi \}.     \lb{1.6} 
\end{equation}
In the half-line context we will employ the notation $L_{p,q,r,s}^{a,D}$ and $L_{p,q,r,s}^{a,N}$ instead of 
$L_{p,q,r,s}$ to emphasize the boundary condition at the finite endpoint $a$; similarly, in the finite interval context we will employ the notation $L_{p,q,r,s}^{(\theta_a,\theta_b)}$. 

Our aim in the real line case is to prove that 
\begin{equation} 
\dom\big(L_{p,q,r,s}^{1/2}\big) = \dom\big(\big(L_{p,q,r,s}^*\big)^{1/2} \big) = W^{1,2}(\bbR).     \lb{1.7} 
\end{equation}
Similarly, in the half-line context we will prove that
\begin{align}
\begin{split} 
\dom\big((L_{p,q,r,s}^{a,D})^{1/2} \big) = \dom\big(\big((L_{p,q,r,s}^{a,D})^*\big)^{1/2} \big) 
= W^{1,2}_0((a,\infty)),&\lb{1.7a}\\
\dom\big((L_{p,q,r,s}^{a,N})^{1/2} \big) = \dom\big(\big((L_{p,q,r,s}^{a,N})^*\big)^{1/2} \big) 
= W^{1,2}((a,\infty)).&
\end{split} 
\end{align}
In the finite interval case we will prove that
\begin{align} 
\begin{split} 
& \dom \Big(\big(L_{p,q,r,s}^{(\theta_a,\theta_b)}\big)^{1/2}\Big) = 
\dom \Big(\big(\big(L_{p,q,r,s}^{(\theta_a,\theta_b)}\big)^*\big)^{1/2}\Big)     \\
& \quad = \dom\big(\mathfrak{q}_{p,q,r,s}^{(\theta_a,\theta_b)} \big) 
= \dom\big(\mathfrak{q}_{1,0,0,0}^{(\theta_a,\theta_b)} \big),  
\quad \theta_a,\theta_b\in \{0,\pi/2\},  \lb{1.7b}
\end{split} 
\end{align}
and if $p=1$ a.e.\ on $(a,b)$, that 
\begin{align} 
\begin{split} 
& \dom \Big(\big(L_{p,q,r,s}^{(\theta_a,\theta_b)}\big)^{1/2}\Big) 
= \dom \Big(\big(\big(L_{p,q,r,s}^{(\theta_a,\theta_b)}\big)^*\big)^{1/2}\Big)     \lb{1.8} \\
& \quad = \dom\big(\mathfrak{q}_{p,q,r,s}^{(\theta_a,\theta_b)} \big) 
= \dom\big(\mathfrak{q}_{1,0,0,0}^{(\theta_a,\theta_b)} \big), \quad \theta_a,\theta_b\in \bbS_{\pi}. 
\end{split} 
\end{align}
Here the quadratic forms $\mathfrak{q}_{1,0,0,0}^{(\theta_a,\theta_b)}$ and 
$\mathfrak{q}_{p,q,r,s}^{(\theta_a,\theta_b)}$ are defined in \eqref{3.3}--\eqref{3.9} and \eqref{3.11}, 
respectively. 

Previously, one-dimensional operators such as $L_{p,q,r,s}$, $L_{p,q,r,s}^{(\theta_a,\theta_b)}$ 
have been considered, typically, under the assumption of uniform ellipticity, and supposing that 
$p,q,r,s\in L^{\infty}((a,b);dx)$, see, for instance, \cite{AMN97a}, \cite{AMN97}, \cite{AT91}, 
\cite{AT92}, \cite{AKM06}, \cite{AKM06a}, \cite{Jo91}, \cite{KM85}, \cite{Mc68}, \cite{Ok80}. 
The case of one-dimensional Schr\"odinger operators on $(0,1)$ and integrable potential $q$ was 
considered in \cite{KR99}. Applying the approach developed in \cite{GHN13}, we now show that the boundedness assumption on the coefficient functions, $q$, $r$, and $s$, can be weakened to the 
general conditions \eqref{1.3} and \eqref{1.4}, respectively. 

To motivate the interest in stability questions of square root domains and especially, relations such as \eqref{1.7} and 
\eqref{1.8}, we now digress a bit and briefly turn to certain aspects of the Kato square root problem. We also mention our final Remark \ref{r1.3} below which offers additional motivations in terms of trace formulas and Fredholm determinants. 

Suppose that $\cH$ denotes a complex Hilbert space with inner product $(\, \cdot\, , \, \cdot \,)_{\cH}$, and let $\mathfrak{q}$ denote a {\it densely defined}, {\it closed}, {\it sectorial} sesquilinear form 
(cf.\ \cite[Sect.\ VI.1]{Ka80})
with domain $\dom(\mathfrak{q})\subseteq\cH$.  Under these assumptions, it is a well-known fact that one can uniquely associate to $\mathfrak{q}$ an m-sectorial operator $T_{\mathfrak{q}}$ for which
\begin{equation}\lb{1.9}
\gq(g,f) = (g,T_{\gq}f)_{\cH},\quad f\in \dom\big(T_{\gq} \big),\, g\in \dom(\gq).
\end{equation}
In this case, $T_{\mathfrak{q}}$ is called the m-sectorial operator associated to $\mathfrak{q}$. 

Any m-accretive operator $T$ (cf.,~e.g., \cite[Sect.\ III.6]{EE89} and 
\cite[Sect.\ V.3.10]{Ka80}) admits a fractional power $T^{\alpha}$, $\alpha\in (0,1)$, such that $\dom(T)$ is a core for $T^{\alpha}$. In this case one has the integral representation
\begin{equation}\lb{1.11}
T^{\alpha}f=\frac{\sin(\pi \alpha)}{\pi}\int_0^{\infty}t^{\alpha-1}T(T+t I_{\cH})^{-1}f\, dt, 
\quad f\in \dom(T), \; \alpha \in (0,1).
\end{equation}
The operator $T^{\alpha}$, $\alpha\in (0,1)$, is itself m-accretive and 
\begin{align}\lb{1.12}
& (T^{\alpha})^*=(T^*)^{\alpha},\\ 
& T^{\alpha}T^{\beta} f =T^{\alpha+\beta}f, \quad f \in \dom\big(T^{\alpha+\beta}\big), 
\; \alpha, \beta, \, \alpha + \beta \in (0,1).
\end{align}
Kato studied fractional powers of m-accretive operators in \cite{Ka61} and isolated the power $\alpha=1/2$ as critical in the following sense. For positive fractional powers below $1/2$, Kato proved in 1961 that 
\begin{equation}\lb{1.13}
\dom(T^{\alpha})=\dom((T^*)^{\alpha}),\quad 0<\alpha<1/2,
\end{equation}
for any m-accretive operator $T$.  By constructing explicit counter examples, Kato showed in the same paper that the equality in \eqref{1.13} need not hold when $1/2<\alpha<1$.  However, the question of 
\begin{equation}\lb{1.14}
\dom(T^{1/2})\stackrel{{\boldsymbol ?}}{=}\dom((T^*)^{1/2})
\end{equation}
for a general m-accretive operator $T$ was left unsettled by Kato \cite{Ka61}.
The following example further illustrates the criticality of the power $\alpha=1/2$.  

\begin{example} \lb{e1.1}
Suppose $p\in L^{\infty}(\bbR:dx)$ and $p\geq \varepsilon$ for some $\varepsilon >0$.
Consider the operator $S$ 
\begin{equation}\lb{1.15}
S:=-\f{d}{dx}p\f{d}{dx}, \quad \dom(S)=\{f\in W^{1,2}(\bbR)\, |\, pf'\in W^{1,2}(\bbR)\},
\end{equation}
in $L^2(\bbR;dx)$. Then $S$ is self-adjoint. The domain of every $\alpha$-th power, $0<\alpha<1$, of $S$ 
is known $($cf., e.g., \cite{AT92}, \cite{DJS85}, \cite[p.~257]{Ou05}$)$, 
\begin{align}
& \dom(S^{\alpha}) = W^{2\alpha,2}(\bbR),\quad 0<\alpha<1/2,\no\\
& \dom(S^{1/2}) = W^{1,2}(\bbR),\lb{1.16}\\
& \dom(S^{\alpha}) = \{f\in W^{1,2}(\bbR)\, |\, pf'\in W^{2\alpha-1,2}(\bbR)\},\quad 1/2<\alpha<1.\no
\end{align}
One notes the sudden change in the domains once the critical power $\alpha=1/2$ is reached.
\end{example}

A year after the appearance of the paper \cite{Ka61} by Kato, Lions \cite{Li62} in 1962 published his construction of an m-accretive operator, the operator associated with 
\begin{equation} 
T = d/dx, \quad \dom(T) = W^{1,2}_0((0,\infty)) 
\end{equation} 
in $L^2((0,\infty); dx)$) for which equality {\it does not} hold in \eqref{1.14} (see also \cite{Mc70} for additional material on such counter examples). This settled the question \eqref{1.14} for m-accretive operators in the negative.  The operator constructed by Lions does not arise as the operator associated to a sectorial sesquilinear form.  Thus, Lions' example left open the possibility of general equality in \eqref{1.14} for an m-sectorial operator $T$ arising as the operator associated to a sectorial sesquilinear form.  

The importance of whether or not
\begin{equation}\lb{1.18}
\dom\big(T_{\mathfrak{q}}^{1/2} \big)=\dom\big((T_{\mathfrak{q}}^*)^{1/2} \big) \, \text{ with $T_{\mathfrak{q}}$
m-sectorial and associated to $\mathfrak{q}$}  
\end{equation}
is brought to light by the following result which was independently proved by Kato and Lions in 1962.

\begin{theorem} \lb{t1.2}$($Kato \cite{Ka62}, Lions \cite{Li62}$)$
Suppose that $\mathfrak{q}$ is a densely defined, closed, sectorial sesquilinear form with domain $\dom(\mathfrak{q})\subset \cH$ and associated m-sectorial operator $T_{\mathfrak{q}}$.  If two of $\dom(T_{\mathfrak{q}}^{1/2})$, $\dom((T_{\mathfrak{q}}^*)^{1/2})$, and $\dom(\mathfrak{q})$ coincide, then all three are equal, and one has a second, complete, representation theorem for $\mathfrak{q}$ in terms of $T_{\mathfrak{q}}^{1/2}$,
\begin{equation}\lb{1.19}
\mathfrak{q}_{{}_T}(f,g)=\big((T^*)^{1/2}f,T^{1/2}g\big)_{\cH},\quad f,g\in \dom(\mathfrak{q})=\dom\big(T_{\mathfrak{q}}^{1/2}\big)=\dom\big((T_{\mathfrak{q}}^*)^{1/2}\big).
\end{equation}
\end{theorem}

We should mention that an extension of Theorem \ref{t1.2} is discussed in 
\cite[Theorem~8.2]{Ou05}.

One notes that when $\mathfrak{q}$ is symmetric and nonnegative, the associated operator $T_{\mathfrak{q}}$ is actually self-adjoint; hence, \eqref{1.18} holds since $T_{\mathfrak{q}}=T_{\mathfrak{q}}^*$.  In this case, \eqref{1.19} is called the {\it 2nd Representation Theorem} \cite[Theorem VI.2.23]{Ka80} for densely defined, closed, nonnegative, symmetric sesquilinear forms.

The question of whether or not \eqref{1.18} holds remained open for some time.  It was not until 1972 that McIntosh \cite{Mc72} was able to construct an m-sectorial operator  for which \eqref{1.18} fails. McIntosh's example showed that a representation of the form \eqref{1.19} is not possible for sectorial sesquilinear forms in general.  Therefore, the best one can do is identify certain examples for which \eqref{1.18} holds.  

Since the example presented by McIntosh was not a differential operator, the question arose as to whether or not \eqref{1.19} holds in the case of second-order elliptic partial differential operators and became known as the Kato square root problem for elliptic partial differential operators.  In this 
context we refer to \cite{CMM82}, \cite{Fu67}, \cite{KM85}, \cite{Mc82}, \cite{Mc85}, 
\cite{Mc90}, \cite{Mi91}, where second-order elliptic partial differential operators of the type 
\begin{equation}
-\sum_{1\leq j,k\leq n} \partial_j a_{j,k} (x) \partial_k + \sum_{1 \leq j \leq n} b_j (x) \partial_j  
+ \sum_{1 \leq j \leq n} \partial_j (c_j (x) \, \cdot \,) + d(x),      \lb{1.20} 
\end{equation}
are discussed in $L^2(\Omega)$, $\Omega \subseteq \bbR^n$ open, under mild 
regularity assumptions on $\partial \Omega$ and the coefficients $a_{j,k}$, $b_j$, $c_j$, and $d$ 
(or certain symmetry hypotheses on these coefficients). The problem continued its great attraction 
throughout the 1990s as is evidenced by the efforts in \cite{Al91}, \cite{AO12}, 
\cite{AMN97}, \cite{AT91}, \cite{AT92}, \cite{AT96},  
\cite{ER97}, \cite{Gu92}, \cite{Jo91}, \cite{Ya84}, \cite{Ya87}, culminating in the treatise 
by Auscher and Tchamitchian \cite{AT98} (see also \cite{AT01}). The final breakthrough and 
the complete solution of Kato's square root problem for elliptic operators of the type 
$- {\rm div}(a\nabla \, \cdot \,)$ on $\bbR^n$ with $L^\infty$-coefficients 
$a_{j,k}$, that is, without any smoothness hypotheses on $a_{j,k}$, occurred in 2001 in papers 
by Auscher, Hofmann, Lacey, Lewis, McIntosh, and Tchamitchian \cite{AHLLMT01}, 
\cite{AHLMT02}, \cite{AHLT01}, \cite{AHMT01}, \cite{HLM02}, \cite{HM02} (see also \cite{Au02}, \cite{Ho01}, 
\cite{Ho01a}). For subsequent developments, including higher-order operators, operators on 
Lipschitz domains $\Omega \subset \bbR^n$, $L^p$-estimates, 
and mixed boundary conditions, leading up to recent work in this area, we also refer to 
\cite{Au04}, \cite{Au07}, \cite{AAM10}, \cite{ABHDR12}, \cite{AT03}, \cite{AKM06}, \cite{BtEM12}, 
\cite{CUR12}, \cite{HM03}, \cite{HMP08}, \cite{Mo12}, \cite[Ch.\ 8]{Ou05}, 
\cite[Sect.\ 2.8, Ch.\ 16]{Ya10}.  

The case $- {\rm div}(a\nabla \, \cdot \,) + V$ in $L^2(\Omega; d^n x)$, $n \in \bbN$, $n \geq 2$, 
with Dirichlet, Neumann, and mixed boundary 
conditions on $\partial \Omega$, under varying assumptions on $\Omega$ (from just an open 
subset of $\bbR^n$ in the Dirichlet case to $\Omega \subseteq \bbR^n$ a strongly Lipschitz 
domain in the case of Neumann or mixed boundary conditions) recently appeared in \cite{GHN13}. 
The conditions on $V$ are of the type $V \in L^p(\Omega) + L^\infty(\Omega)$ with $p > n/2$.  The 
case of critical (i.e., optimal) $L^p$-conditions resulting in a proof of the fact that the square root 
domain associated with 
\begin{equation}\label{1.21}
- {\rm div}(a\nabla \, \cdot \,)  + \big(\B_1 \cdot \nabla \cdot \big) + \dv \big(\B_2 \cdot \big) + V,
\end{equation}
in $L^2(\bbR^n)$ under the conditions 
\begin{equation} 
a \in L^\infty(\bbR^n)^{n \times n}, \quad 
\B_1,\B_2 \in L^n(\bbR^n)^n + L^\infty(\bbR^n)^n, \quad V\in L^{n/2}(\bbR^n) + L^\infty(\bbR^n),    
\lb{1.22} 
\end{equation} 
is given by the standard Sobolev space $W^{1,2}(\bbR^n)$ for $n \geq 3$ is also given in \cite{GHN13}. 
This reference also contains a very detailed bibliography on this circle of ideas in dimensions $n \geq 2$. 
In addition, \cite{GHN13} contains abstract results implying 
\begin{equation}
\dom\big((T_0 + W)^{1/2}\big) = \dom\big(T_0^{1/2}\big),    \lb{1.23}
\end{equation}
as well as 
\begin{equation}
\dom\big((T_0 + W)^{1/2}\big) = \dom\big(T_0^{1/2}\big) = \dom\big((T_0^*)^{1/2}\big) 
= \dom\big(((T_0 + W)^*)^{1/2}\big).     \lb{1.24}
\end{equation}
under the assumption that 
\begin{equation}
\dom\big(T_0^{1/2}\big) = \dom\big((T_0^*)^{1/2}\big)    \lb{1.25} 
\end{equation}
for $T_0$ an m-accretive operator and appropriate perturbations $W$. While \cite{GHN13} primarily treated 
the case $n \geq 2$, we now exclusively focus on the one-dimensional setting, $n=1$ (cf.\ 
\eqref{1.1}--\eqref{1.8}.

We conclude these introductory remarks with one more remark illustrating the importance for 
control of square root domains: 

\begin{remark} \lb{r1.3}
The stability of square root domains has important consequences for trace formulas 
involving resolvent differences and symmetrized Fredholm (perturbation) determinants as 
discussed in detail in \cite{GZ12}: 
Let $A$ and $A_0$ be densely defined, closed, linear operators in $\cH$. 
Suppose there exists $t_0 \in \bbR$ such that $A + t_0 I_{\cH}$ and
$A_0 + t_0 I_{\cH}$ are of positive-type and $(A  + t_0 I_{\cH}) \in \Sect(\omega_0)$,
$(A_0  + t_0 I_{\cH}) \in \Sect(\omega_0)$ for some $\omega_0 \in [0,\pi)$. 
In addition, assume that for some $t_1 \geq t_0$,
\begin{align}
& \dom\big((A_0 + t_1 I_{\cH})^{1/2}\big) \subseteq
\dom\big((A + t_1 I_{\cH})^{1/2}\big),     \lb{7.18} \\
& \dom\big((A_0^* + t_1 I_{\cH})^{1/2}\big) \subseteq
\dom\big((A^* + t_1 I_{\cH})^{1/2}\big),     \lb{7.19} \\
& \ol{(A + t_1 I_{\cH})^{1/2} \big[(A + t_1 I_{\cH})^{-1} - (A_0 + t_1 I_{\cH})^{-1}\big]
	(A + t_1 I_{\cH})^{1/2}} \in \cB_1(\cH).    \lb{7.20}
\end{align}
Then 
\begin{align}
\begin{split}
& - \f{d}{dz} \ln\Big({\det}_{\cH}\Big(\ol{(A - z I_{\cH})^{1/2}(A_0 - z I_{\cH})^{-1}
	(A - z I_{\cH})^{1/2}}\Big)\Big)   \\
&\quad = {\tr}_{\cH}\big((A - z I_{\cH})^{-1} - (A_0 - z I_{\cH})^{-1}\big),     \lb{7.21}
\end{split}
\end{align}
for all $z \in \bbC \backslash \big(\ol{- t_0 + S_{\omega_0}}\big)$ such that 
$\ol{(A - z I_{\cH})^{1/2}(A_0 - z I_{\cH})^{-1}(A - z I_{\cH})^{1/2}}$ is boundedly invertible. 
(Analogous trace formulas can be derived for modified Fredholm determinants replacing 
$\cB_1(\cH)$ by $\cB_k(\cH)$, $k \in \bbN$, $k \geq 2$, including additional terms under 
the trace.) 

Here $A$ is said to be of {\it positive-type} if
\begin{align}
(-\infty, 0] \subset \rho(A) \, \text{ and } \, 
M_A = \sup_{t \geq 0} \big\|(1 + t) (A + t I_{\cH})^{-1}\big\|_{\cB(\cH)} < \infty,   \lb{7.7} 
\end{align}
and $A$ is called {\it sectorial of angle $\omega \in [0,\pi)$}, denoted by
$A \in \Sect(\omega)$, if
\begin{align}
\sigma(A) \subseteq \ol{S_{\omega}}  \, 
\text{ and for all $\omega' \in (\omega,\pi)$, }  \,
M(A,\omega') = \sup_{z\in \bbC\backslash \ol{S_{\omega'}}}
\big\| z (A - z I_{\cH})^{-1}\big\|_{\cB(\cH)} < \infty,       \lb{7.8} 
\end{align}
where $S_{\omega} \subset \bbC$, $\omega \in [0,\pi)$, denotes the open sector
\begin{equation}
S_{\omega} = \begin{cases} \{z\in\bbC \,|\, z \neq 0, \, |\arg(z)|<\omega\},
& \omega \in (0,\pi), \\
(0,\infty), & \omega = 0,
\end{cases}
\end{equation}
with vertex at $z=0$ along the positive real axis and opening angle $2 \omega$. Assumptions 
such as \eqref{7.18}, \eqref{7.19} make it plain that control over square root domains is at the core 
for the validity of the trace formula \eqref{7.21}.  

In the special case where in addition $A$ and $A_0$ are self-adjoint in $\cH$, trace relations of 
the type \eqref{7.21} are intimately related to the notion of the spectral shift function associated 
with the pair $(A,A_0)$, and hence underscore its direct relevance to spectral and scattering 
theory for this pair (cf.\ \cite{GZ12} for details). 
\end{remark}

Next, we briefly turn to a description of the content of each section: Section \ref{s2} contains 
our results in the real line case. Section \ref{s4} is devoted to the half-line $(a,\infty)$, with 
Dirichlet or Neumann boundary conditions at $a \in \bbR$. Our final Section \ref{s3} treats the 
finite interval case $(a,b)$ with Dirichlet and/or Neumann boundary conditions at $a,b \in \bbR$, 
and in the particular case $p=1$ a.e.\ on $(a,b)$, all possible separated boundary conditions 
at $a, b$ are treated. Appendices \ref{sA}--\ref{sC} provide auxiliary results on operator, Trudinger, 
and form bounds used in the bulk of this paper.   

Finally, we denote by $AC([a,b])$ the set of absolutely continuous functions on the interval 
$[a,b] \subset \bbR$, by $W^{1,2}((c,d))$ the usual Sobolev space, and by 
\begin{equation}
L^p_{\locunif}(\bbR; dx) = \bigg\{f \in L^p_{\loc}(\bbR; dx) \,\bigg|\,
\sup_{a\in\bbR} \bigg(\int_a^{a+1} dx \, |f(x)|^p\bigg) < \infty\bigg\}, \quad
p \in [1,\infty),       \lb{2.1a}
\end{equation}
the spaces of locally uniformly $L^p$-integrable functions on $\bbR$.

\section{Sturm--Liouville Operators in $L^2(\bbR;dx)$}  \lb{s2}

\begin{hypothesis}\lb{h2.1}
Suppose that $p$, $q$, $r$, and $s$ satisfy the following conditions: \\
$(i)$ $p:\bbR\rightarrow \bbC$ and there exist constants $0<\lambda\leq \Lambda$, such that 
\begin{equation}\lb{2.1}
\lambda<\Re(p(x)) \, \text{ and } \, |p(x)|\leq \Lambda \, \text{ for a.e.\ $x\in \bbR$.} 
\end{equation}
$(ii)$  $q\in L^1_{\locunif}(\bbR;dx)$ and $r,\, s \in L^2_{\locunif}(\bbR;dx)$.
\end{hypothesis}

Note that $p$, $q$, $r$, and $s$ are not assumed to be real-valued.  Moreover, $q$, $r$, and $s$ need {\it not} be bounded.

Assuming Hypothesis \ref{h2.1}, define sesquilinear forms $\mathfrak{q}_j$, $j\in \{0,1,2,3\}$, as follows:
\begin{align}
& \mathfrak{q}_0(f,g) = \int_{\bbR}dx\, \overline{f'(x)}p(x)g'(x),\lb{2.2}\\
& \mathfrak{q}_1(f,g) = \int_{\bbR}dx\, \overline{f(x)}r(x)g'(x),\lb{2.3}\\
& \mathfrak{q}_2(f,g) = \int_{\bbR}dx\, \overline{f'(x)}s(x)g(x),\lb{2.4}\\
& \quad f,g\in \dom(\mathfrak{q}_0)=\dom(\mathfrak{q}_1)=\dom(\mathfrak{q}_2)=W^{1,2}(\bbR),\no\\
& \mathfrak{q}_3(f,g) = \int_{\bbR}dx\, \overline{f(x)}q(x)g(x),\lb{2.5}\\
& \quad f,g\in \dom(\mathfrak{q}_3)=\big\{h\in L^2(\bbR;dx)\, \big|\, |q|^{1/2}h\in L^2(\bbR;dx) \big\}.\no
\end{align}

Throughout this manuscript, we follow the terminology and conventions for sesquilinear forms set forth in \cite[Ch.~6]{Ka80}.  One notes that the sesquilinear form $\mathfrak{q}_0$ is actually a closed sectorial sesquilinear form.

\begin{proposition}\lb{p2.2}
Assume Hypothesis \ref{h2.1} and let $\mathfrak{q}_j$, $j\in \{0,1,2,3\}$, defined as in \eqref{2.2}--\eqref{2.5}.  The following items hold:\\
$(i)$  $\mathfrak{q}_0$ is a densely defined, sectorial, and closed sesquilinear form.\\
$(ii)$  Each of the sesquilinear forms $\mathfrak{q}_j$, $j\in \{1,2,3\}$, is infinitesimally bounded with respect to $\mathfrak{q}_0$.  In particular, there exist constants $\e_0 > 0$ and $M > 0$ such that
\begin{equation}\lb{2.6}
\begin{split}
|\mathfrak{q}_j(f,f)|\leq \e \Re(\mathfrak{q}_0(f,f)) + M\e^{-3}\|f\|_{L^2(\bbR;dx)}^2,&\\
f\in W^{1,2}(\bbR), \; 0<\e<\e_0, \, j\in \{1,2,3\}.&
\end{split}
\end{equation}
\end{proposition}
\begin{proof}
Evidently, $\mathfrak{q}_0$ is densely defined as $W^{1,2}(\bbR)$ is dense in $L^2(\bbR;dx)$.  To show that $\mathfrak{q}_0$ is sectorial, one notes that the boundedness assumption on $p$ in \eqref{2.1} implies
\begin{equation}\lb{2.7}
|\Im(\mathfrak{q}_0(f,f))|\leq \int_{\bbR}dx\, |\Im(p(x))||f'(x)|^2\leq \Lambda \int_{\bbR}dx\, |f'(x)|^2, \quad f\in W^{1,2}(\bbR),
\end{equation}
while the first inequality in \eqref{2.1} implies
\begin{equation}\lb{2.8}
\int_{\bbR}dx\, |f'(x)|^2\leq \lambda^{-1}\Re(\mathfrak{q}_0(f,f)), \quad f\in W^{1,2}(\bbR).
\end{equation}
The inequalities in \eqref{2.7} and \eqref{2.8} combine to yield
\begin{equation}\lb{2.9}
|\Im(\mathfrak{q}_0(f,f))|\leq\frac{\Lambda}{\lambda}\, \Re(\mathfrak{q}_0(f,f)),\quad f\in W^{1,2}(\bbR),
\end{equation}
and it follows that $\mathfrak{q}_0$ is sectorial.

To show that $\mathfrak{q}_0$ is closed requires one to prove that $\dom(\mathfrak{q}_0)=W^{1,2}(\bbR)$, equipped with the norm $\|\cdot\|_{\mathfrak{q}_0}$ defined by
\begin{equation}\lb{2.10}
\|f\|_{\mathfrak{q}_0}^2=\Re(\mathfrak{q}_0(f,f))+\|f\|_{L^2(\bbR;dx)}^2,\quad f\in \dom(\mathfrak{q}_0)=W^{1,2}(\bbR),
\end{equation}
is closed.  Under the assumption of Hypothesis \ref{h2.1}$(i)$, $\|\cdot\|_{\mathfrak{q}_0}$ and $\|\cdot\|_{W^{1,2}(\bbR)}$ are actually equivalent norms.  In fact, using \eqref{2.1}, one establishes the following elementary estimate:
\begin{equation}\lb{2.11}
\begin{split}
\min\{\lambda,1\}\|f\|_{W^{1,2}(\bbR)}^2\leq \Re(\mathfrak{q}_0(f,f))+\|f\|_{L^2(\bbR;dx)}^2\leq \max\{\Lambda,1\} \|f\|_{W^{1,2}(\bbR)}^2,&\\
 f\in \dom(\mathfrak{q}_0)=W^{1,2}(\bbR).&
\end{split}
\end{equation}
Consequently, $W^{1,2}(\bbR)$ is closed with respect to $\|\cdot\|_{\mathfrak{q}_0}$ (since it is closed with respect to $\|\cdot\|_{W^{1,2}(\bbR)}$, and these two norms are equivalent).  As a result, $\mathfrak{q}_0$ is closed, and the proof of item $(i)$ is complete.

In order to prove item $(ii)$, one recalls the fundamental result that any complex-valued function belonging to $L^2_{\locunif}(\bbR;dx)$ is infinitesimally bounded with respect to $d/dx$ defined on $W^{1,2}(\bbR)$ 
in $L^2(\bbR;dx)$ (cf., e.g., \cite[Theorem\ 2.7.1]{Sc81}, \cite[p.\ 35]{Si71}). In fact, we note in passing, this condition is not only sufficient, but also necessary for relative (as well as, relative infinitesimal) boundedness  (see also the survey in \cite{GW14}).  In particular, since 
$|q|^{1/2},\, r,\, s\in L^2_{\locunif}(\bbR;dx)$, \cite[Theorem 2.7.1]{Sc81} and \cite[p.\ 35]{Si71} imply the following form bounds:
\begin{align}
\|rf\|_{L^2(\bbR;dx)}^2&\leq \kappa C_r\|f'\|_{L^2(\bbR;dx)}^2+2\kappa^{-1}C_r\|f\|_{L^2(\bbR;dx)}^2,\lb{2.12}\\
\|sf\|_{L^2(\bbR;dx)}^2&\leq \kappa C_s\|f'\|_{L^2(\bbR;dx)}^2+2\kappa^{-1}C_s\|f\|_{L^2(\bbR;dx)}^2,\lb{2.13}\\
\big\||q|^{1/2}f\big\|_{L^2(\bbR;dx)}^2&\leq \kappa C_q\|f'\|_{L^2(\bbR;dx)}^2+2\kappa^{-1}C_q\|f\|_{L^2(\bbR;dx)}^2,\lb{2.14}\\
&\hspace*{2.5cm} f\in W^{1,2}(\bbR), \; 0<\kappa<1,\no
\end{align}
where
\begin{equation}\lb{2.15}
C_q=\sup_{a\in \bbR}\int_a^{a+1}dx\, |q(x)|,\quad C_r=\sup_{a\in \bbR}\int_a^{a+1}dx\, |r(x)|^2,\quad C_s=\sup_{a\in \bbR}\int_a^{a+1}dx\, |s(x)|^2.
\end{equation}
We assume without loss that $C_rC_sC_q\neq 0$; otherwise, the problem simplifies.  Subsequently, \eqref{2.12}--\eqref{2.14} imply
\begin{align}
\|rf\|_{L^2(\bbR;dx)}&\leq \e \|f'\|_{L^2(\bbR;dx)}+C_0\e^{-1}\|f\|_{L^2(\bbR;dx)},\lb{2.16}\\
\|sf\|_{L^2(\bbR;dx)}&\leq \e \|f'\|_{L^2(\bbR;dx)}+C_0\e^{-1}\|f\|_{L^2(\bbR;dx)},\lb{2.17}\\
\big\||q|^{1/2}f\big\|_{L^2(\bbR;dx)}^2&\leq \e\|f'\|_{L^2(\bbR;dx)}^2+C_0\e^{-1}\|f\|_{L^2(\bbR;dx)}^2,\lb{2.18}\\
&\hspace*{1.3cm}f\in W^{1,2}(\bbR),\; 0<\e<\e_{q,r,s},\no
\end{align}
where
\begin{equation}\lb{2.19}
C_0=\sqrt{2}\cdot \max\big\{1,C_r,C_s,\sqrt{2}C_q^2\big\} \, \text{ and } 
\, \e_{q,r,s}=\min\big\{\sqrt{C_r},\sqrt{C_s},C_q \big\}.
\end{equation}
The estimate in \eqref{2.16} (resp., \eqref{2.17}) follows by taking square roots throughout \eqref{2.12} (resp., \eqref{2.13}) and choosing $\e=\sqrt{\kappa C_r}$ (resp., $\e=\sqrt{\kappa C_s}$).  On the other hand, \eqref{2.18} follows from \eqref{2.14} by choosing $\e=\kappa C_q$.  Note the constant $C_0$ is introduced to provide uniformity throughout \eqref{2.16}--\eqref{2.18}, allowing one to dispense with keeping track of the three separate constants $C_q$, $C_r$, and $C_s$.  Consequently, one estimates
\begin{align}
&|\mathfrak{q}_j(f,f)|\no\\
&\quad \leq \big\||q|^{1/2}f\big\|_{L^2(\bbR;dx)}^2+\big|(rf,f')_{L^2(\bbR;dx)}\big|+\big|(f',sf)_{L^2(\bbR;dx)}\big|\no\\
&\quad \leq \big\||q|^{1/2}f\big\|_{L^2(\bbR;dx)}^2+\|rf\|_{L^2(\bbR;dx)}\|f'\|_{L^2(\bbR;dx)}+\|f'\|_{L^2(\bbR;dx)}\|sf\|_{L^2(\bbR;dx)}\no\\
&\quad \leq 3\e\|f'\|_{L^2(\bbR;dx)}^2+C_0\e^{-1}\|f\|_{L^2(\bbR;dx)}^2+2C_0\e^{-1}\|f\|_{L^2(\bbR;dx)}\|f'\|_{L^2(\bbR;dx)},\lb{2.20}\\
&\hspace*{4.85cm}f\in W^{1,2}(\bbR),\; 0<\e<\e_{q,r,s},\, j\in \{1,2,3\}.\no
\end{align}
Using the elementary inequality
\begin{equation}\lb{2.21}
\begin{split}
2C_0\e^{-1}\|f\|_{L^2(\bbR;dx)}\|f'\|_{L^2(\bbR;dx)}\leq \e\|f'\|_{L^2(\bbR;dx)}^2+C_0^2\e^{-3}\|f\|_{L^2(\bbR;dx)}^2,&\\
 f\in W^{1,2}(\bbR), \; \e>0,&
 \end{split}
\end{equation}
which follows from
\begin{equation}\lb{2.22}
\big(\e^{1/2}\|f'\|_{L^2(\bbR;dx)}-C_0\e^{-3/2}\|f\|_{L^2(\bbR;dx)} \big)^2\geq 0, \quad f\in W^{1,2}(\bbR),\; \e>0,
\end{equation}
one continues the estimate in \eqref{2.20}, using \eqref{2.8} to obtain
\begin{align}
|\mathfrak{q}_j(f,f)|&\leq 4\e \|f'\|_{L^2(\bbR;dx)}^2+\big(C_0\e^{-1}+C_0^2\e^{-3}\big)\|f\|_{L^2(\bbR;dx)}^2\no\\
&\leq 4\lambda^{-1}\e\Re(\mathfrak{q}_0(f,f))+\big(C_0\e^{-1}+C_0^2\e^{-3}\big)\|f\|_{L^2(\bbR;dx)}^2,\lb{2.23}\\
&\hspace*{1.85cm}f\in W^{1,2}(\bbR),\; 0<\e<\e_{q,r,s},\, j\in\{1,2,3\}.\no
\end{align}
Subsequently, rescaling $\e$ throughout \eqref{2.23} yields
\begin{align}
|\mathfrak{q}_j(f,f)|& \leq \e\Re(\mathfrak{q}_0(f,f))+\big[4C_0\lambda^{-1}\e^{-1}+64C_0^2\lambda^{-3}\e^{-3} \big]\|f\|_{L^2(\bbR;dx)}^2\no\\
&\leq \e\Re(\mathfrak{q}_0(f,f))+128C_0^2\big(\lambda^{-1}+\lambda^{-3}\big)\e^{-3}\|f\|_{L^2(\bbR;dx)}^2,\lb{2.24}\\
&\hspace*{-.175cm}f\in W^{1,2}(\bbR),\; 0<\e<\min\{1,4\lambda^{-1}\e_{q,r,s}\},\, j\in \{1,2,3\},\no
\end{align}
upon using $C_0<C_0^2$ and $\e^{-1}<\e^{-3}$ if $0<\e<1$.  This completes the proof of item $(ii)$ and establishes the form bound in \eqref{2.6} with, say,
\begin{equation}\lb{2.25}
M=128C_0^2\big(\lambda^{-1}+\lambda^{-3} \big) \, \text{ and } \, \e_0=\min\{1, 4\lambda^{-1}\e_{q,r,s}\}.
\end{equation}
\end{proof}

By Proposition \ref{p2.2} and \cite[Theorem VI.1.33]{Ka80}, the sesquilinear form
\begin{equation}\lb{2.26}
\mathfrak{q}(f,g):=\sum_{j=0}^3\mathfrak{q}_j(f,g), \quad \dom(\mathfrak{q}) = W^{1,2}(\bbR),
\end{equation}
is a densely defined, closed, sectorial sesquilinear form in $L^2(\bbR;dx)$.  Therefore, $\mathfrak{q}$ is uniquely associated to an m-sectorial operator by the 1st representation theorem \cite[Theorem VI.2.1]{Ka80}, and we denote this operator by $L_{p,q,r,s}$.  Formally speaking, $L_{p,q,r,s}$ takes the form
\begin{equation}\lb{2.27}
L_{p,q,r,s} = -\frac{d}{dx}p\frac{d}{dx}+r\frac{d}{dx}-\frac{d}{dx}s+q.
\end{equation}
We recall the following fundamental results:

\begin{theorem}\lb{t2.3}$($\cite[Thm.~3.1]{AMN97}, \cite{KM85}$)$.
If $p\in L^{\infty}(\bbR;dx)$ and $\Re(p(x))>\lambda>0$ for a.e.\ $x\in \bbR$ and some $\lambda>0$, then
\begin{equation}\lb{2.28}
\dom\big(L_{p,0,0,0}^{1/2} \big) = \dom\big(\big(L_{p,0,0,0}^*\big)^{1/2} \big) = W^{1,2}(\bbR).
\end{equation} 
\end{theorem}

Theorem \ref{t2.3} is used as a basic input in the principal result of this section, Theorem \ref{t2.5} below, which states that the square root domain of $L_{p,q,r,s}$ is actually independent of the coefficients $p$, $q$, $r$, and $s$, provided these coefficients satisfy the assumptions in Hypothesis \ref{h2.1}.

\begin{theorem}\lb{t2.5}
Assume Hypothesis \ref{h2.1}, then
\begin{equation}\lb{2.33}
\dom\big(L_{p,q,r,s}^{1/2}\big) = \dom\big(\big(L_{p,q,r,s}^*\big)^{1/2} \big) = W^{1,2}(\bbR).
\end{equation} 
\end{theorem}
\begin{proof}
It suffices to establish $\dom\big(L_{p,q,r,s}^{1/2}\big) = W^{1,2}(\bbR)$ for all $p$, $q$, $r$, and $s$ satisfying the assumptions in Hypothesis \ref{h2.1}.  Then by taking complex conjugates of $p$, $q$, $r$, and $s$, we obtain $\dom\big(\big(L_{p,q,r,s}^*\big)^{1/2} \big)= W^{1,2}(\bbR)$, implying \eqref{2.33}.  We carry out the proof in a two-step process and begin by proving the special case when $s=0$ a.e.\ on $\bbR$.  That is, we will first show that
\begin{equation}\lb{2.35}
\dom\big(L_{p,q,r,0}^{1/2} \big)=W^{1,2}(\bbR),
\end{equation}
and then prove the statement for general $s$ satisfying the assumptions in Hypothesis \ref{h2.1}. 

One recalls that $L_{p,q,r,0}$ was defined via sesquilinear forms.  Alternatively, one can define $L_{p,q,r,0}$ indirectly in terms of its resolvent by 
\begin{align}
&\big(L_{p,q,r,0}-zI_{L^2(\bbR;dx)}\big)^{-1}=\big(L_{p,0,0,0}-zI_{L^2(\bbR;dx)}\big)^{-1}   \no \\
&\quad-\ol{\big(L_{p,0,0,0}-zI_{L^2(\bbR;dx)}\big)^{-1}B_1^*}\bigg[I_{L^2(\bbR;dx)}+\ol{A_1\big(L_{p,0,0,0}-zI_{L^2(\bbR;dx)}\big)^{-1}B_1^*} \bigg]^{-1}\no\\
&\qquad \times A_1\big(L_{p,0,0,0}-zI_{L^2(\bbR;dx)}\big)^{-1},    \lb{2.36} \\
&\qquad \; z\in \Big\{\zeta\in \rho\big(L_{p,0,0,0}\big)\, \Big| \, 1\in \rho\Big(\ol{A_1\big(L_{p,0,0,0}-\zeta I_{L^2(\bbR;dx)}\big)^{-1}B_1^*} \Big)\Big\},\no
\end{align}
where the operators $A_1$ and $B_1$ are defined by 
\begin{align}
\begin{split}\lb{2.37}
&A_1:L^2(\bbR;dx)\rightarrow L^2(\bbR;dx)^2,\quad A_1f=\begin{pmatrix}
f' \\
e^{i\text{Arg}(q)}|q|^{1/2}f
\end{pmatrix},\\
&B_1:L^2(\bbR;dx)\rightarrow L^2(\bbR;dx)^2,\quad B_1f=\begin{pmatrix}
\ol{r}f \\
|q|^{1/2}f
\end{pmatrix},
\end{split}\\
& \quad f\in \dom(A_1)=\dom(B_1)=W^{1,2}(\bbR).\no
\end{align}
Here, 
\begin{equation}
\begin{split}
L^2(\bbR;dx)^2&=\big\{ f = (f_1\quad f_2)^{\top}\, \big|\, f_j\in L^2(\bbR;dx),\, j=1,2\big\},\\
\|f\|_{L^2(\bbR;dx)^2}^2 &= \|f_1\|_{L^2(\bbR;dx)}^2 + \|f_2\|_{L^2(\bbR;dx)}^2,\quad f = (f_1\quad f_2)^{\top}\in L^2(\bbR;dx )^2.
\end{split}
\end{equation}
The identity in \eqref{2.36}, which we call {\it Kato's resolvent identity for $L_{p,q,r,0}$ and $L_{p,0,0,0}$}, is obtained by viewing $L_{p,q,r,0}$ as an additive perturbation of $L_{p,0,0,0}$ by a term which can be factored as $B_1^*A_1$ over the auxiliary Hilbert space $L^2(\bbR;dx)^2$.  Indeed, one notes that the operators $A_1$ and $B_1$ defined in \eqref{2.37} are closed, densely defined, linear operators from $L^2(\bbR;dx)$ into $L^2(\bbR;dx)^2$, with
\begin{equation}\lb{2.38}
\dom\big(L_{p,0,0,0} \big)\subset \dom(A_1) \, \text{ and } \, \dom\big(L_{p,0,0,0}^*\big)\subset \dom(B_1).
\end{equation}
Moreover, an application of \eqref{2.31} with $S=A_1$ and $T=L_{p,0,0,0}$, yields the estimate
\begin{equation}\lb{2.39}
\big\|A_1(L_{p,0,0,0}+EI_{L^2(\bbR;dx)})^{-1/2}\big\|_{\cB(L^2(\bbR;dx),L^2(\bbR;dx)^2)}\leq C_1, \quad E>E_0,
\end{equation}
for an $E$-independent constant, $C_1>0$ and a constant $E_0>0$.  Subsequently, another application of Lemma \ref{l2.4}, along with Theorems \ref{t2.3} and \ref{ta.1}, provides the estimate
\begin{align}
&\big\|B_1(L_{p,0,0,0}^*+EI_{L^2(\bbR;dx)})^{-1/2}f\big\|_{L^2(\bbR;dx)^2}\no\\
& \quad \leq \Big\{\big\||r|(L_{1,0,0,0}+EI_{L^2(\bbR;dx)})^{-1/2} \big\|_{\cB(L^2(\bbR;dx))} \no\\
&\qquad+\big\||q|^{1/2}(L_{1,0,0,0}+EI_{L^2(\bbR;dx)})^{-1/2} \big\|_{\cB(L^2(\bbR;dx))}\Big\}\no\\
&\quad \quad  \times M_1\|f\|_{L^2(\bbR;dx)}\no\\
& \quad \leq C_2E^{-1/4}\|f\|_{L^2(\bbR;dx)},\quad f\in L^2(\bbR;dx),\; E>E_0',\lb{2.40}
\end{align}
where $C_2>0$ is an $E$-independent constant, $M_1>0$ is a constant for which
\begin{equation}\lb{2.41}
\begin{split}
\big\|(L_{1,0,0,0}+EI_{L^2(\bbR;dx)})^{1/2}(L_{p,0,0,0}^*+EI_{L^2(\bbR;dx)})^{-1/2} \big\|_{\cB(L^2(\bbR;dx))}\leq M_1,&\\
E\geq 1,&
\end{split}
\end{equation}
and $E_0'$ is the constant guaranteed to exist by Theorem \ref{ta.1}.
Consequently, \eqref{2.40} implies the norm bound
\begin{equation}\lb{2.42}
\big\|B_1(L_{p,0,0,0}^*+EI_{L^2(\bbR;dx)})^{-1/2}\big\|_{\cB(L^2(\bbR;dx),\cB(L^2(\bbR;dx)^2)}\leq C_2E^{-1/4},\quad E>E_0',
\end{equation}
and it follows that 
\begin{align}
&\Big\|\ol{A_1(L_{p,0,0,0}+EI_{L^2(\bbR;dx)})^{-1}B_1^*}\Big\|_{\cB(L^2(\bbR;dx)^2)}\no\\
&\quad \leq \big\|A_1(L_{p,0,0,0}+EI_{L^2(\bbR;dx)})^{-1/2}\big\|_{\cB(L^2(\bbR;dx),L^2(\bbR;dx)^2)}\no\\
&\quad \quad \times \Big\|\ol{(L_{p,0,0,0}+EI_{L^2(\bbR;dx)})^{-1/2}B_1^*}
\Big\|_{\cB(L^2(\bbR;dx)^2,L^2(\bbR;dx))}\no\\
&\quad \leq C_1C_2E^{-1/4}, \quad E>\max\{E_0,E_0'\}.\lb{2.43}
\end{align}
As a result, one infers that
\begin{equation}\lb{2.44}
1\in \rho\big(\ol{A_1(L_{p,0,0,0}+EI_{L^2(\bbR;dx)})^{-1}B_1^*} \big),\quad E>\max\{E_0,E_0'\},
\end{equation}
so that the set appearing at the end of \eqref{2.36} is nonempty. By \cite[Theorem 2.3]{GLMZ05}, the expression on the right-hand 
side of the equality in \eqref{2.36} is the resolvent of a densely defined, closed, linear operator in $L^2(\bbR;dx)$. 
This operator coincides with $L_{p,q,r,0}$ by Theorem \ref{te.3}. In view of \eqref{2.43} and 
the fact that
\begin{equation}\lb{2.45}
\dom\big(L_{p,0,0,0}^{1/2} \big)=\dom(A_1)=\dom\big(\big(L_{p,0,0,0}^* \big)^{1/2}\big)=\dom(B_1),
\end{equation}
the statement
\begin{equation}\lb{2.46}
\dom\big(L_{p,q,r,0}^{1/2} \big)=\dom\big(L_{p,0,0,0}^{1/2} \big),
\end{equation}
follows from an application of Corollary \ref{cd.5a}, and \eqref{2.35} follows from Theorem \ref{t2.3}.  This completes the proof of the result in the special case $s=0$ a.e.\ on $\bbR$.

In order to establish the general case, we next prove
\begin{equation}\lb{2.47}
\dom\big(L_{p,q,r,s}^{1/2} \big) = \dom\big(L_{p,q,r,0}^{1/2} \big).
\end{equation}
The claim $\dom\big(L_{p,q,r,s}^{1/2}\big) = W^{1,2}(\bbR)$ will then follow from \eqref{2.35} and \eqref{2.46}.  In order 
to show \eqref{2.47}, we note that a resolvent identity similar to \eqref{2.36} holds between $L_{p,q,r,s}$ 
and $L_{p,q,r,0}$.  More specifically, 
\begin{align}
&\big(L_{p,q,r,s}-zI_{L^2(\bbR;dx)}\big)^{-1}=\big(L_{p,q,r,0}-zI_{L^2(\bbR;dx)}\big)^{-1}    \no \\
&\quad-\ol{\big(L_{p,q,r,0}-zI_{L^2(\bbR;dx)}\big)^{-1}B_2^*}\bigg[I_{L^2(\bbR;dx)}+\ol{A_2\big(L_{p,q,r,0}-zI_{L^2(\bbR;dx)}\big)^{-1}B_2^*} \bigg]^{-1}\no\\
&\qquad \times A_2\big(L_{p,q,r,0}-zI_{L^2(\bbR;dx)}\big)^{-1},    \lb{2.48} \\
&\hspace*{1.2cm}z\in \big\{\zeta\in \rho\big(L_{p,q,r,0}\big)\, \big| \, 1\in \rho\big(\ol{A_2(L_{p,q,r,0}-\zeta I_{L^2(\bbR;dx)})^{-1}B_2^*} \big)\big\},\no
\end{align}
where the operators $A_2,B_2:L^2(\bbR;dx)\rightarrow L^2(\bbR;dx)$ are defined by 
\begin{equation}\lb{2.49}
A_2f=-sf, \quad B_2f=f',\quad
\dom(A_2)=\dom(B_2)=W^{1,2}(\bbR).
\end{equation}
Indeed, one can again prove a decay estimate for $A_2$, $B_2$, and $L_{p,q,r,0}$ of the type \eqref{2.43}, namely, 
\begin{align}
&\Big\|\ol{A_2(L_{p,q,r,0}+EI_{L^2(\bbR;dx)})^{-1}B_2^*}\Big\|_{\cB(L^2(\bbR;dx))}\no\\
&\quad \leq \big\|A_2(L_{p,q,r,0}+EI_{L^2(\bbR;dx)})^{-1/2}\big\|_{\cB(L^2(\bbR;dx))}\no\\
&\quad \quad \times \Big\|\ol{(L_{p,q,r,0}+EI_{L^2(\bbR;dx)})^{-1/2}B_2^*}
\Big\|_{\cB(L^2(\bbR;dx))}\no\\
&\quad \leq CE^{-\alpha}, \quad E>E_0,\lb{2.46.0}
\end{align}
for appropriate constants $C>0$, $\alpha>0$, and $E_0>0$ to show
\begin{equation}\lb{2.50}
1\in \rho\big(\ol{A_2(L_{p,q,r,0}+EI_{L^2(\bbR;dx)})^{-1}B_2^*} \big), 
\end{equation}
for $E > 0$ sufficiently large. The calculations involved in establishing a bound of the type \eqref{2.46.0} are similar to those carried out in \eqref{2.39}--\eqref{2.43}, but simplify slightly since the factorization $B_2^*A_2$ is carried out over $L^2(\bbR;dx)$ (instead of $L^2(\bbR;dx)^2$, as in the case of the factorization $B_1^*A_1$), so we omit them 
here. As a result, by \cite[Theorem 2.3]{GLMZ05} the right-hand side of \eqref{2.48} defines the resolvent of a densely defined, closed, linear operator.  The operator so defined coincides with $L_{p,q,r,s}$ by Theorem \ref{te.3}.  Since
\begin{equation}\lb{2.51}
\dom\big(L_{p,q,r,0}^{1/2} \big)=\dom(A_2)=\dom\big(\big(L_{p,q,r,0}^* \big)^{1/2}\big)=\dom(B_2),
\end{equation}
another application of Corollary \ref{cd.5a} yields \eqref{2.47}, completing the proof.
\end{proof}

\begin{remark}
In order to apply the abstract results of Corollary \ref{cd.5a}, the two-step process in the proof of 
Theorem \ref{t2.5} is critical. Instead of the two-step process, one might instead try to replace $L_{p,q,r,0}$ 
in \eqref{2.36} by $L_{p,q,r,s}$ and prove $\dom\big(\big(L_{p,q,r,s}^*\big)^{1/2} \big) 
= \dom\big( L_{p,q,r,s}^{1/2}\big)$ directly. The problem with this approach is that $L_{p,q,r,s}$ is a 
perturbation of $L_{p,0,0,0}$ by a term which formally factors as $B_3^*A_3$ over the auxiliary Hilbert 
space $L^2(\bbR;dx)^3$ in the following way:
\begin{equation}\lb{2.52}
\begin{split}
&A_3,B_3:L^2(\bbR;dx)\rightarrow L^2(\bbR;dx)^3,\quad \dom\big(A_3 \big)=\dom\big(B_3 \big)=W^{1,2}(\bbR)\\
& A_3f=\begin{pmatrix}
f' \\
-sf\\
e^{i\text{Arg}(q)}|q|^{1/2}f
\end{pmatrix},\quad B_3f=\begin{pmatrix}
\ol{r}f \\
f'\\
|q|^{1/2}f
\end{pmatrix},\quad f\in W^{1,2}(\bbR).
\end{split}
\end{equation}
The difficulty with this approach is that simple techniques for estimating norms (like the first inequality 
in \eqref{2.40}, for example) are not sharp enough to yield decay of
\begin{align}
&\big\|B_3(L_{p,0,0,0}^*+EI_{L^2(\bbR;dx)})^{-1/2}\big\|_{\cB(L^2(\bbR;dx),L^2(\bbR;dx)^3)},\lb{2.53}\\
& \big\|A_3(L_{p,0,0,0}+EI_{L^2(\bbR;dx)})^{-1/2}\big\|_{\cB(L^2(\bbR;dx),L^2(\bbR;dx)^3)},\lb{2.54}
\end{align}
in terms of an inverse power of $E$.  Indeed, following such a simple approach to estimate the operator 
norm in \eqref{2.53}, one is faced with the norm,
\begin{equation}\lb{2.55}
\big\|(d/dx)(L_{p,0,0,0}+EI_{L^2(\bbR;dx)})^{-1/2} \big\|_{\cB(L^2(\bbR;dx))}, \quad E>0.
\end{equation}
It is clear that the norm in \eqref{2.55} is uniformly bounded for all $E\geq 1$ (cf. Lemma \ref{l2.4} and 
\eqref{2.31}), but one cannot expect decay as $E\rightarrow \infty$, and certainly not as an inverse 
power of $E$. The end result is that, with crude estimates such as these, one only shows that the 
norms in \eqref{2.53} and \eqref{2.54} are {\it uniformly bounded} in $E\geq 1$.  Boundedness without 
sufficient decay in $E$ is not sufficient to apply Corollary \ref{cd.5a}.
\end{remark}

\section{Sturm--Liouville Operators in $L^2((a,\infty);dx)$}  \lb{s4}

\begin{hypothesis}\lb{h4.1}
Let $a\in \bbR$ be fixed.  Suppose that $p$, $q$, $r$, and $s$ satisfy the following conditions: \\
$(i)$ $p:(a,\infty)\rightarrow \bbC$ and there exist constants $0<\lambda\leq \Lambda$, such that 
\begin{equation}\lb{4.1}
\lambda<\Re(p(x)) \, \text{ and } \, |p(x)|\leq \Lambda \, \text{ for a.e.\ $x\in (a,\infty)$.} 
\end{equation}
$(ii)$  $q\in L^1_{\locunif}([a,\infty);dx)$ and $r,\, s \in L^2_{\locunif}([a,\infty);dx)$ in the sense that
\begin{align}
&N_q:=\sup_{c\in [a,\infty)}\bigg( \int_c^{c+1} dx\, |q(x)|\bigg)<\infty,\quad N_r:=\sup_{c\in [a,\infty)}\bigg(\int_c^{c+1}dx\, |r(x)|^2\bigg)<\infty,\no\\
&\hspace*{2.7cm} N_s:=\sup_{c\in [a,\infty)}\bigg(\int_c^{c+1}dx\, |s(x)|^2\bigg)<\infty.\lb{4.2}
\end{align}
In addition, fix $\cW((a,\infty))$ with either $\cW((a,\infty))=W^{1,2}((a,\infty))$ or $\cW((a,\infty))=W^{1,2}_0((a,\infty))$.
\end{hypothesis}

Again, note that $p$, $q$, $r$, and $s$ are not assumed to be real-valued and that $q$, $r$, and $s$ need {\it not} be bounded.

Assuming Hypothesis \ref{h4.1}, define sesquilinear forms $\mathfrak{q}_j$, $j\in \{0,1,2,3\}$, as follows:
\begin{align}
& \mathfrak{q}_0(f,g) = \int_a^{\infty}dx\, \overline{f'(x)}p(x)g'(x),\lb{4.3}\\
& \mathfrak{q}_1(f,g) = \int_a^{\infty}dx\, \overline{f(x)}r(x)g'(x),\lb{4.4}\\
& \mathfrak{q}_2(f,g) = \int_a^{\infty}dx\, \overline{f'(x)}s(x)g(x),\lb{4.5}\\
& \quad f,g\in \dom(\mathfrak{q}_0)=\dom(\mathfrak{q}_1)=\dom(\mathfrak{q}_2)=\cW((a,\infty)),\no\\
& \mathfrak{q}_3(f,g) = \int_a^{\infty}dx\, \overline{f(x)}q(x)g(x),\lb{4.6}\\
& \quad f,g\in \dom(\mathfrak{q}_3)=\big\{h\in L^2((a,\infty);dx)\, \big|\, |q|^{1/2}h\in L^2((a,\infty);dx) \big\}.\no
\end{align}

The following proposition is a necessary first step in order to use the sesquilinear forms $\mathfrak{q}_j$, $j\in \{0,1,2,3\}$, to define an m-sectorial operator.
\begin{proposition}\lb{p4.2}
Assume Hypothesis \ref{h2.1} and let $\mathfrak{q}_j$, $j\in \{0,1,2,3\}$, be defined as in \eqref{4.2}--\eqref{4.5}. Then the following items hold:\\
$(i)$  $\mathfrak{q}_0$ is a densely defined, sectorial, and closed sesquilinear form.\\
$(ii)$  Each of the sesquilinear forms $\mathfrak{q}_j$, $j\in \{1,2,3\}$, is infinitesimally bounded with respect to $\mathfrak{q}_0$.  In particular, there exist constants $\e_0 > 0$ and $M > 0$ such that
\begin{equation}\lb{4.7}
\begin{split}
|\mathfrak{q}_j(f,f)|\leq \e \Re(\mathfrak{q}_0(f,f)) + M\e^{-3}\|f\|_{L^2((a,\infty);dx)}^2,&\\
f\in \cW((a,b)), \; 0<\e<\e_0, \, j\in \{1,2,3\}.&
\end{split}
\end{equation}
\end{proposition}
\begin{proof}
The proof of item $(i)$ is essentially identical to the proof of item $(i)$ in Proposition \ref{p2.2}, one simply replaces $W^{1,2}(\bbR)$ by $\cW((a,\infty))$ and the integration over $\bbR$ by integration over $(a,\infty)$ in \eqref{2.7}--\eqref{2.8}, so we omit further details.

In order to prove item $(ii)$, we define
\begin{align}
\Omega_j=(a+j,a+j+1), \quad j\in \bbN\cup\{0\}.\lb{4.8}
\end{align}

If $\phi_j$ denotes the restriction of any function $\phi: (a,\infty)\rightarrow \bbC$ to $\Omega_j$,
\begin{align}
\phi_j = \phi|_{\Omega_j}, \quad j\in \bbN\cup \{0\},\lb{4.9}
\end{align}
then an application of Theorem \ref{tb.1} implies
\begin{align}
\begin{split} 
\|w_j f_j\|_{L^2(\Omega_j;dx)}^2 \leq \e N \| (f_j)'\|_{L^2(\Omega_j;dx)}^2+\big(1+\e^{-1}\big)N\|f_j\|_{L^2(\Omega_j;dx)}^2,&\lb{4.10}\\
f\in \cW((a,\infty)), w\in \big\{|q|^{1/2},r,s\big\},\, j\in \bbN\cup\{0\},\, \e>0,& 
\end{split} 
\end{align}
where we have set
\begin{align}
N = \max \big\{N_q, N_r, N_s, 1 \big\}.\lb{4.11}
\end{align}
In light of the trivial fact,
\begin{align}
\|f\|_{L^2((a,\infty);dx)}^2 = \sum_{j=0}^{\infty} \|f_j\|_{L^2(\Omega_j;dx)}^2,\quad f\in L^2((a,\infty);dx),\lb{4.12}
\end{align}
\eqref{4.10} immediately implies
\begin{align}
\begin{split} 
\|wf\|_{L^2((a,\infty);dx)}^2\leq \e N \|f'\|_{L^2((a,\infty);dx)}^2+2\e^{-1}N\|f\|_{L^2((a,\infty);dx)}^2,&\lb{4.13}\\
f\in \cW((a,\infty)),\, w\in \big\{|q|^{1/2},r,s \big\},\, 0<\e<1,&   
\end{split}
\end{align}
by summing over all $\Omega_j$ and making use of
\begin{align}
(f_j)'(x)=(f')_j(x)\, \text{ a.e. $x\in \Omega_j$},\, j\in \bbN\cup\{0\},\, f\in \cW((a,\infty)).\lb{4.14}
\end{align}
Now one can show that \eqref{4.7} follows from \eqref{4.13} by mimicking, with only minor modifications, the same strategy used to deduce \eqref{2.6} from \eqref{2.12}--\eqref{2.14}.  We omit further details at this point.
\end{proof}

With the sesquilinear forms $\mathfrak{q}_j$, $j\in \{0,1,2,3\}$, in hand, we are now ready to define an m-sectorial operator.  By Proposition \ref{p4.2} and \cite[Theorem VI.1.33]{Ka80}, the sesquilinear form
\begin{equation}\lb{4.15}
\mathfrak{q}(f,g):=\sum_{j=0}^3\mathfrak{q}_j(f,g), \quad \dom(\mathfrak{q}) = \cW((a,\infty)),
\end{equation}
is a densely defined, closed, sectorial sesquilinear form in $L^2((a,\infty);dx)$.  Therefore, $\mathfrak{q}$ is uniquely associated to an m-sectorial operator by the 1st representation theorem \cite[Theorem VI.2.1]{Ka80}, and we denote this operator by $L_{p,q,r,s}^{a,N}$ when $\cW((a,\infty))=W^{1,2}((a,\infty))$ and by $L_{p,q,r,s}^{a,D}$ when $\cW((a,\infty))=W^{1,2}_0((a,\infty))$.  Formally speaking, $L_{p,q,r,s}^{a,N}$ and $L_{p,q,r,s}^{a,D}$ take the form
\begin{equation}
L_{p,q,r,s}^{a,Y} = -\frac{d}{dx}p\frac{d}{dx}+r\frac{d}{dx}-\frac{d}{dx}s+q, \quad Y\in \{D,N\}.\lb{4.16}
\end{equation}
Here, the superscript {\it D} denotes {\it Dirichlet} and {\it N} denotes {\it Neumann} since functions in the domain of $L_{p,q,r,s}^{a,D}$ (resp., $L_{p,q,r,s}^{a,N}$) satisfy a Dirichlet (resp., Neumann) boundary condition at $x=a$ of the form $f(a)=0$ (resp., $f^{[1]}(a)=0$, where $y^{[1]}=py' \in AC([a,R])$ denotes the first 
quasi-derivative of $y \in AC([a,R])$ for all $R>a$).

In the simple case when $q=r=s=0$ a.e. in $(a,\infty)$, one obtains coincidence of square root domains as a special case of the results of \cite{AMN97}, \cite{AT92}.

\begin{theorem}\lb{t4.3}$($\cite[Thm.~6.1 and Case I\,$(ii)$ on p.~685]{AMN97}$;$ \cite[(2.2b) and (2.2c)]{AT92}$)$.
Let $a\in \bbR$.  If $p\in L^{\infty}((a,\infty);dx)$ and $\Re(p(x))>\lambda>0$ for a.e.\ $x\in (a,\infty)$ and some $\lambda>0$, then
\begin{align}
\dom\big((L_{p,0,0,0}^{a,D})^{1/2} \big) = \dom\big(\big((L_{p,0,0,0}^{a,D})^*\big)^{1/2} \big) = W^{1,2}_0((a,\infty)),&\lb{4.17}\\
\dom\big((L_{p,0,0,0}^{a,N})^{1/2} \big) = \dom\big(\big((L_{p,0,0,0}^{a,N})^*\big)^{1/2} \big) = W^{1,2}((a,\infty)).\lb{4.18}&
\end{align} 
\end{theorem}

Using Theorem \ref{t4.3} as a basic input, we can now state and prove the main result of this section.

\begin{theorem}\lb{t4.4}
Assume Hypothesis \ref{h4.1}, then
\begin{align}
\dom\big((L_{p,q,r,s}^{a,D})^{1/2} \big) = \dom\big(\big((L_{p,q,r,s}^{a,D})^*\big)^{1/2} \big) = W^{1,2}_0((a,\infty)),&\lb{4.19}\\
\dom\big((L_{p,q,r,s}^{a,N})^{1/2} \big) = \dom\big(\big((L_{p,q,r,s}^{a,N})^*\big)^{1/2} \big) = W^{1,2}((a,\infty)).\lb{4.20}&
\end{align}
\end{theorem}
\begin{proof}
Fix $Y\in \{D,N\}$ and denote the domain of the sesquilinear form associated to $L_{p,q,r,s}^{a,Y}$ by $\cW((a,\infty))$.  Specifically, $\cW((a,\infty))=W^{1,2}((a,\infty)$ if $Y=N$ and $\cW((a,\infty)) = W^{1,2}_0((a,\infty))$ if $Y=D$.  In order to obtain \eqref{4.19} and \eqref{4.20}, it suffices to establish $\dom\big((L_{p,q,r,s}^{a,Y})^{1/2}\big) = \cW((a,\infty))$ for all $p$, $q$, $r$, and $s$ satisfying the assumptions in Hypothesis \ref{h4.1}.  Then by taking complex conjugates of $p$, $q$, $r$, and $s$, we obtain $\dom\big(\big((L_{p,q,r,s}^{a,Y})^*\big)^{1/2} \big)= \cW((a,\infty))$, and \eqref{4.19}, \eqref{4.20} then follow.  We carry out the proof in a two-step process and begin by proving the result in the special case when $s=0$ a.e.\ on $(a,\infty)$.  That is, we will first show that
\begin{equation}\lb{4.21}
\dom\big((L_{p,q,r,0}^{a,Y})^{1/2} \big)=\cW((a,\infty)),
\end{equation}
and then prove the statement for general $s$ satisfying the assumptions of Hypothesis \ref{h4.1}.

Although we defined $L_{p,q,r,0}^{a,Y}$ via sesquilinear forms, we may alternatively define it indirectly in terms of its resolvent by applying Kato's resolvent identity (cf.\ Theorem \ref{te.3}),
\begin{align}
&\big(L_{p,q,r,0}^{a,Y}-zI_{L^2((a,\infty);dx)}\big)^{-1}=\big(L_{p,0,0,0}^{a,Y}-zI_{L^2((a,\infty);dx)}\big)^{-1} 
\no \\
&\quad-\ol{\big(L_{p,0,0,0}^{a,Y}-zI_{L^2((a,\infty);dx)}\big)^{-1}B_1^*}\no\\
&\quad \times \bigg[I_{L^2((a,\infty);dx)}+\ol{A_1\big(L_{p,0,0,0}^{a,Y}-zI_{L^2((a,\infty);dx)}\big)^{-1}B_1^*} \bigg]^{-1}\no\\
&\qquad \times A_1\big(L_{p,0,0,0}^{a,Y}-zI_{L^2((a,\infty);dx)}\big)^{-1},    \lb{4.22} \\
&\qquad \; z\in \Big\{\zeta\in \rho\big(L_{p,0,0,0}^{a,Y}\big)\, \Big| \, 1\in \rho\Big(\ol{A_1\big(L_{p,0,0,0}^{a,Y}-\zeta I_{L^2((a,\infty);dx)}\big)^{-1}B_1^*} \Big)\Big\},\no
\end{align}
where the operators $A_1$ and $B_1$ are defined by 
\begin{align}
\begin{split}\lb{4.23}
&A_1:L^2((a,\infty);dx)\rightarrow L^2((a,\infty);dx)^2,\quad A_1f=\begin{pmatrix}
f' \\
e^{i\text{Arg}(q)}|q|^{1/2}f
\end{pmatrix},\\
&B_1:L^2((a,\infty);dx)\rightarrow L^2((a,\infty);dx)^2,\quad B_1f=\begin{pmatrix}
\ol{r}f \\
|q|^{1/2}f
\end{pmatrix},
\end{split}\\
& \quad f\in \dom(A_1)=\dom(B_1)=\cW((a,\infty)).\no
\end{align}
The identity in \eqref{4.22} is obtained by viewing $L_{p,q,r,0}^{a,Y}$ as an additive perturbation of $L_{p,0,0,0}^{a,Y}$ by a term which can be factored as $B_1^*A_1$ over the auxiliary Hilbert space $L^2((a,\infty);dx)^2$.  Indeed, one notes that the operators $A_1$ and $B_1$ defined in \eqref{4.23} are closed, densely defined, linear operators from $L^2((a,\infty);dx)$ into $L^2((a,\infty);dx)^2$, with
\begin{equation}\lb{4.24}
\dom\big(L_{p,0,0,0}^{a,Y} \big)\subset \dom(A_1) \, \text{ and } \, \dom\big((L_{p,0,0,0}^{a,Y})^*\big)\subset \dom(B_1),
\end{equation}
and an application of \eqref{2.31} with $S=A_1$ and $T=L_{p,0,0,0}^{a,Y}$, yields the uniform estimate
\begin{equation}\lb{4.25}
\big\|A_1\big(L_{p,0,0,0}^{a,Y}+EI_{L^2((a,\infty);dx)}\big)^{-1/2}\big\|_{\cB(L^2((a,\infty);dx),L^2((a,\infty);dx)^2)}\leq C_1, \quad E>E_0,
\end{equation}
for an $E$-independent constant $C_1>0$ and a constant $E_0>0$.  In turn, Lemma \ref{l2.4} and Theorems \ref{t4.3} and \ref{ta.2} imply
\begin{align}
\begin{split} 
&\big\|B_1\big((L_{p,0,0,0}^{a,Y})^*+EI_{L^2((a,\infty);dx)}\big)^{-1/2}f\big\|_{L^2((a,\infty);dx)^2}    \\
& \quad \leq C_2E^{-1/4}\|f\|_{L^2((a,\infty);dx)},\quad f\in L^2((a,\infty);dx),\; E>E_0',\lb{4.26}
\end{split} 
\end{align}
where $C_2$ is an $E$-independent constant, and $E_0'$ is the constant guaranteed to exist by Theorem \ref{a.2}.  One arrives at \eqref{4.26} by employing a similar strategy to the one used to obtain \eqref{2.40}.  As a result, one infers the norm bound
\begin{align}
& \big\|B_1\big((L_{p,0,0,0}^{a,Y})^*+EI_{L^2((a,\infty);dx)}\big)^{-1/2}\big\|_{\cB(L^2((a,\infty);dx),L^2((a,\infty);dx)^2)}\leq C_2E^{-1/4},    \no \\
& \hspace*{9.8cm} E>E_0'.    \lb{4.27} 
\end{align}
Subsequently, \eqref{4.25} and \eqref{4.27} imply
\begin{align}
&\Big\|\ol{A_1\big(L_{p,0,0,0}^{a,Y}+EI_{L^2((a,\infty);dx)}\big)^{-1}B_1^*}\Big\|_{\cB(L^2((a,\infty);dx)^2)}\no\\
&\quad \leq \big\|A_1\big(L_{p,0,0,0}^{a,Y}+EI_{L^2((a,\infty);dx)}\big)^{-1/2}\big\|_{\cB(L^2((a,\infty);dx),L^2((a,\infty);dx)^2)}\no\\
&\quad \quad \times \Big\|\ol{\big(L_{p,0,0,0}^{a,Y}+EI_{L^2((a,\infty);dx)}\big)^{-1/2}B_1^*}
\Big\|_{\cB(L^2((a,\infty);dx)^2,L^2((a,\infty);dx))}\no\\
&\quad \leq C_1C_2E^{-1/4}, \quad E>\max\{E_0,E_0'\},\lb{4.28}
\end{align}
so that
\begin{equation}\lb{4.29}
1\in \rho\Big(\ol{A_1\big(L_{p,0,0,0}^{a,Y}+EI_{L^2((a,\infty);dx)}\big)^{-1}B_1^*} \Big),\quad E>\max\{E_0,E_0'\}.
\end{equation}
As a result, the set appearing in the right-hand side of \eqref{4.22} is nonempty.  By \cite[Theorem 2.3]{GLMZ05}, the right-hand side of \eqref{4.22} is the resolvent of a densely defined, closed, linear operator in $L^2((a,\infty);dx)$.  This operator coincides with $L_{p,q,r,0}^{a,Y}$ by Theorem \ref{te.3}.  Next, the estimate in \eqref{4.28} along with the fact that
\begin{equation}\lb{4.30}
\dom\Big(\big(L_{p,0,0,0}^{a,Y}\big)^{1/2} \Big)=\dom(A_1)=\dom\Big(\big(\big(L_{p,0,0,0}^{a,Y}\big)^* \big)^{1/2}\Big)=\dom(B_1),
\end{equation}
together imply 
\begin{equation}\lb{4.31}
\dom\big((L_{p,q,r,0}^{a,Y})^{1/2} \big)=\dom\big((L_{p,0,0,0}^{a,Y})^{1/2} \big),
\end{equation}
applying Corollary \ref{cd.5a}.  Finally, the statement in \eqref{4.21} follows from \eqref{4.31} and Theorem \ref{4.3}.

The next step in the proof is to establish \eqref{4.19} and \eqref{4.20} for arbitrary $s$ satisfying the assumptions of Hypothesis \ref{h4.1}.  This is carried out by showing that 
\begin{equation}\lb{4.32}
\dom\big((L_{p,q,r,s}^{a,Y})^{1/2} \big) = \dom\big((L_{p,q,r,0}^{a,Y})^{1/2} \big).
\end{equation}
Subsequently, the claim that $\dom\big((L_{p,q,r,s}^{a,Y})^{1/2}\big) = \cW((a,\infty))$ will follow from \eqref{4.21}.

In order to show \eqref{4.32}, we use a resolvent identity similar to \eqref{4.22}, but this time for the resolvents of $L_{p,q,r,s}^{a,Y}$ and $L_{p,q,r,0}^{a,Y}$:
\begin{align}
&\big(L_{p,q,r,s}^{a,Y}-zI_{L^2((a,\infty);dx)}\big)^{-1}=\big(L_{p,q,r,0}^{a,Y}-zI_{L^2((a,\infty);dx)}\big)^{-1}   
\no \\
&\quad-\ol{\big(L_{p,q,r,0}^{a,Y}-zI_{L^2((a,\infty);dx)}\big)^{-1}B_2^*}\no\\
&\qquad \times \bigg[I_{L^2((a,\infty);dx)}+\ol{A_2\big(L_{p,q,r,0}^{a,Y}-zI_{L^2((a,\infty);dx)}\big)^{-1}B_2^*} \bigg]^{-1}\no\\
&\qquad \times A_2\big(L_{p,q,r,0}^{a,Y}-zI_{L^2((a,\infty);dx)}\big)^{-1},    \lb{4.33} \\
&\qquad \; z\in \Big\{\zeta\in \rho\big(L_{p,q,r,0}^{a,Y}\big)\, \Big| \, 1\in \rho\Big(\ol{A_2\big(L_{p,q,r,0}^{a,Y}-\zeta I_{L^2((a,\infty);dx)}\big)^{-1}B_2^*} \Big)\Big\},\no
\end{align}
where, in analogy to \eqref{2.49}, the operators $A_2,B_2:L^2((a,\infty);dx)\rightarrow L^2((a,\infty);dx)$ are defined by 
\begin{equation}\lb{4.34}
A_2f=-sf, \quad B_2f=f',\quad
\dom(A_2)=\dom(B_2)=\cW((a,\infty)).
\end{equation}
To verify that \eqref{4.33} is correct, one uses a strategy analogous to the one used to establish the similar identities \eqref{2.36}, \eqref{2.48}, and \eqref{4.22}.  The main step is to prove a decay estimate for $A_2$, $B_2$, and $L_{p,q,r,0}^{a,Y}$ of the type
\begin{align}
&\Big\|\ol{A_2(L_{p,q,r,0}^{a,Y}+EI_{L^2((a,\infty);dx)})^{-1}B_2^*}\Big\|_{\cB(L^2((a,\infty);dx))}\no\\
&\quad \leq \big\|A_2(L_{p,q,r,0}^{a,Y}+EI_{L^2((a,\infty);dx)})^{-1/2}\big\|_{\cB(L^2((a,\infty);dx))}\no\\
&\quad \quad \times \Big\|\ol{(L_{p,q,r,0}^{a,Y}+EI_{L^2((a,\infty);dx)})^{-1/2}B_2^*}
\Big\|_{\cB(L^2((a,\infty);dx))}\no\\
&\quad \leq CE^{-\alpha}, \quad E>E_0,\lb{3.35.0}
\end{align}
for appropriate constants $C>0$, $\alpha>0$, and $E_0>0$ in order to show that 
\begin{equation}\lb{4.35}
1\in \rho\Big(\ol{A_2\big((L_{p,q,r,0}^{a,Y})+EI_{L^2((a,\infty);dx)}\big)^{-1}B_2^*} \Big), 
\end{equation}
for $E > 0$ sufficiently large. The calculations involved in proving an estimate of the form \eqref{3.35.0} are similar to those in 
\eqref{4.25}--\eqref{4.28}, but simplify slightly since the factorization $B_2^*A_2$ is carried out over 
$L^2((a,\infty);dx)$ (instead of $L^2((a,\infty);dx)^2$, as in the case of the factorization $B_1^*A_1$), 
so we omit them here. Hence, \cite[Theorem 2.3]{GLMZ05} implies the right-hand side of \eqref{4.33} 
is the resolvent of a densely defined, closed, linear operator in $L^2((a,\infty);dx)$.  The operator so 
defined coincides with $L_{p,q,r,s}^{a,Y}$ by Theorem \ref{te.3}. Since
\begin{equation}\lb{4.36}
\dom\Big((L_{p,q,r,0}^{a,Y})^{1/2} \Big)=\dom(A_2)
=\dom\Big(\big((L_{p,q,r,0}^{a,Y})^* \big)^{1/2}\Big)=\dom(B_2),
\end{equation}
yet another application of Corollary \ref{cd.5a} yields \eqref{4.32}.
\end{proof}

We note that in the special case where $p=1$ a.e.\ on $(a,\infty)$, one can use an approach based 
on Krein's resolvent formula to prove an extension of Theorem \ref{t4.4} to include all non-self-adjoint 
boundary conditions at the endpoint $a$,
\begin{equation}
\cos(\theta_a)g(a)+\sin(\theta_a)g'(a) = 0, \quad \theta_a \in \bbS_{\pi}.
\end{equation}
Since we will demonstrate this approach in detail in the final interval context in Lemma \ref{l3.3} 
and Theorem \ref{t3.6}, we omit the analogous discussion in the present half-line case.

\section{Sturm--Liouville Operators in $L^2((a,b);dx)$}  \lb{s3}

\begin{hypothesis}\lb{h3.1}
Suppose $-\infty<a<b<\infty$ and that $p$, $q$, $r$, and $s$ satisfy the following conditions: \\
$(i)$  $p:(a,b)\rightarrow \bbC$ and there exist constants $0<\lambda\leq \Lambda$ such that 
\begin{equation}\lb{3.1}
\lambda<\Re(p(x)) \, \text{ and } \, |p(x)|\leq \Lambda \, \text{ for a.e.\ $x\in (a,b)$.}
\end{equation}
$(ii)$  $q\in L^1((a,b);dx)$ and $r,\, s \in L^2((a,b);dx)$.
\end{hypothesis}

We are interested in defining realizations of the formal expression on the right-hand side in \eqref{2.27} in $L^2((a,b);dx)$ with certain boundary conditions at $a$ and $b$.  Denoting by $\bbS_{\pi}$ the strip 
\begin{equation}\lb{3.2}
\bbS_{\pi}=\{z\in \bbC\, |\, 0\leq \Re(z)<\pi \},
\end{equation}
we consider the following sesquilinear forms defined by
\begin{align}
&\mathfrak{q}_p^{(\theta_a,\theta_b)}(f,g)=\int_a^b dx \, \ol{f'(x)}p(x)g'(x) -\cot(\theta_a) \ol{f(a)}g(a)-\cot(\theta_b) \ol{f(b)}g(b),     \no \\
&f,g\in \dom\big(\mathfrak{q}_p^{(\theta_a,\theta_b)} \big)=W^{1,2}((a,b))    \lb{3.3} \\
&\quad=\big\{h\in L^2((a,b);dx)\, \big| \, h\in AC([a,b]),\, h'\in L^2((a,b);dx)\big\},\, \theta_a,\theta_b\in \bbS_{\pi}\backslash\{0\},\no\\[1mm]
&\mathfrak{q}_p^{(0,\theta_b)}(f,g)=\int_a^b dx \, \ol{f'(x)}p(x)g'(x) -\cot(\theta_b) \ol{f(b)}g(b),    \no \\
&f,g\in \dom\big(\mathfrak{q}_p^{(0,\theta_b)} \big)     \lb{3.4} \\
&\quad=\big\{h\in L^2((a,b);dx)\, \big| \, h\in AC([a,b]),\, h'\in L^2((a,b);dx),\, h(a)=0\big\},\no\\
&\hspace*{8.8cm} \theta_b\in \bbS_{\pi}\backslash\{0\},\no\\[1mm]
&\mathfrak{q}_p^{(\theta_a,0)}(f,g)=\int_a^b dx \, \ol{f'(x)}p(x)g'(x) -\cot(\theta_a) \ol{f(a)}g(a),    \no \\
&f,g\in \dom\big(\mathfrak{q}_p^{(\theta_a,0)} \big)    \lb{3.5} \\
&\quad=\{h\in L^2((a,b);dx)\, \big| \, h\in AC([a,b]),\, h'\in L^2((a,b);dx),\, h(b)=0\},\no\\
&\hspace*{8.7cm}\theta_a\in \bbS_{\pi}\backslash\{0\},\no\\[1mm]
&\mathfrak{q}_p^{(0,0)}(f,g)=\int_a^b dx \, \ol{f'(x)}p(x)g'(x),    \no \\
&f,g\in \dom\big(\mathfrak{q}_p^{(0,0)} \big)     \lb{3.6} \\
&\quad =\{h\in L^2((a,b);dx)\, \big| \, h\in AC([a,b]),\, h'\in L^2((a,b);dx),\, h(a)=h(b)=0\}.\no
\end{align}
It is a well-known fact that the sesquilinear forms in \eqref{3.3}--\eqref{3.6} are closed, sectorial, and densely defined in $L^2((a,b);dx)$.  Additionally, assuming Hypothesis \ref{h3.1}, we define
\begin{align}
&\mathfrak{q}_r(f,g)=\int_a^b dx\, \overline{f(x)}r(x)g'(x),\lb{3.7}\\
&\mathfrak{q}_s(f,g)=\int_a^b dx\, \overline{f'(x)}s(x)g(x),\lb{3.8}\\
& \quad f,g\in \dom(\mathfrak{q}_r)=\dom(\mathfrak{q}_s)=W^{1,2}((a,b)),\no\\
&\mathfrak{q}_q(f,g)=\int_a^b dx\, \overline{f(x)}q(x)g(x),\lb{3.9}\\
& \quad f,g\in \dom(\mathfrak{q}_q)=\big\{h\in L^2((a,b);dx)\, \big|\, |q|^{1/2}h\in L^2((a,b);dx) \big\}.\no
\end{align}

\begin{proposition}\lb{p3.2}
Assume Hypothesis \ref{h3.1}.  Then there exist constants $C_0>0$ and $\e_0>0$ such that
\begin{equation}\lb{3.10}
\begin{split}
& \big|\mathfrak{q}_{\phi}(f,f)\big|\leq \e\Re\big(\mathfrak{q}_p^{(\theta_a,\theta_b)}(f,f) \big)+C_0\e^{-3}\|f\|_{L^2((a,b);dx)}^2,   \\
& f\in \dom\big(\mathfrak{q}_p^{(\theta_a,\theta_b)} \big),\; 0<\e<\e_0,\; 
\theta_a,\theta_b\in\big\{0,\tfrac{\pi}{2}\big\},\; \phi\in \{q,r,s\}.  
\end{split}
\end{equation}
Further,
\begin{equation}\lb{3.11}
\begin{split}
\mathfrak{q}_{p,q,r,s}^{(\theta_a,\theta_b)}:=\mathfrak{q}_p^{(\theta_a,\theta_b)}+\mathfrak{q}_r+\mathfrak{q}_s+\mathfrak{q}_q, \quad \dom\big(\mathfrak{q}_{p,q,r,s}^{(\theta_a,\theta_b)}\big)=\dom\big(\mathfrak{q}_p^{(\theta_a,\theta_b)}\big),&\\
\theta_a,\theta_b\in \bbS_{\pi},&
\end{split}
\end{equation}
is a closed, sectorial, and densely defined sesquilinear form.
\end{proposition}
\begin{proof}
Fixing $\phi\in\{q,r,s\}$, applying Theorem \ref{tb.1} (cf.~\eqref{b.2}) and repeating the arguments of \eqref{2.7}--\eqref{2.24}, one infers that in the cases $\theta_a,\theta_b\in \{0,\pi/2\}$, $\mathfrak{q}_{\phi}$ is infinitesimally form bounded with respect to $\mathfrak{q}_p^{(\theta_a,\theta_b)}$ and that a form bound of the type in \eqref{3.10} holds.

For the claim that the form defined in \eqref{3.11} is closed, sectorial, and densely defined, we consider only the case $\theta_a\neq 0$ and $\theta_b\neq 0$; the degenerate cases $\theta_a=0$ or $\theta_b=0$ can be handled similarly.  Then $\mathfrak{q}_{p,q,r,s}^{(\theta_a,\theta_b)}$ may be viewed as a form sum perturbation of $\mathfrak{q}_p^{(\pi/2,\pi/2)}$ by the sum of $\mathfrak{q}_{q}+\mathfrak{q}_{r}+\mathfrak{q}_{s}$ with the following densely defined and closed forms
\begin{equation}\lb{3.12}
\begin{split}
\mathfrak{q}_{\theta_a}(f,g)=-\cot(\theta_a)\ol{f(a)}g(a), \quad \mathfrak{q}_{\theta_b}(f,g)=-\cot(\theta_b)\ol{f(b)}g(b),&\\
\dom(\mathfrak{q}_{\theta_a}) = \dom(\mathfrak{q}_{\theta_b})=AC([a,b]).&
\end{split}
\end{equation}
By Theorem \ref{tb.1}, each of the forms in \eqref{3.12} is infinitesimally bounded with respect to $\mathfrak{q}_p^{(\pi/2,\pi/2)}$.  Consequently, 
\begin{equation}\lb{3.13}
\mathfrak{q}_{q}+\mathfrak{q}_{r}+\mathfrak{q}_{s}+\mathfrak{q}_{\theta_a}+\mathfrak{q}_{\theta_b} 
\, \text{ with domain $AC([a,b])$}
\end{equation}
is infinitesimally bounded with respect to $\mathfrak{q}_p^{(\pi/2,\pi/2)}$ and subsequently 
$\mathfrak{q}_{q,r,s}^{(\theta_a,\theta_b)}$  is closed, sectorial, and densely defined applying 
\cite[Theorem VI.1.33]{Ka80}.  One handles the special cases $\theta_a=0$ or $\theta_b=0$ in a similar manner.
\end{proof}

By the 1st representation theorem \cite[Theorem VI.2.1]{Ka80}, one can uniquely associate to the 
sesquilinear form $\mathfrak{q}_{p,q,r,s}^{(\theta_a,\theta_b)}$ an m-sectorial operator.  We denote this operator by $L_{p,q,r,s}^{(\theta_a,\theta_b)}$.  One notes that (cf.~\eqref{3.3}--\eqref{3.6} and \eqref{3.11})
\begin{equation}\lb{3.14}
\mathfrak{q}_p^{(\theta_a,\theta_b)}=\mathfrak{q}_{p,0,0,0}^{(\theta_a,\theta_b)}, \quad \theta_a,\theta_b\in \bbS_{\pi}.
\end{equation}

In our first result, we isolate the two special cases corresponding to the choices $\theta_a=\theta_b=0$ (i.e., Dirichlet boundary conditions) or $\theta_a=\theta_b=\pi/2$ (i.e., Neumann boundary conditions). 
Before stating our result, we recall the following fundamental result which is a special case of results 
obtained in \cite{AMN97} and \cite{AT92}.

\begin{theorem}\lb{t3.3a}$($\cite[Thm.~6.1, Cases I\,$(iii)$,~I\,$(iv)$ on p.~685]{AMN97}$;$ \cite[(2.2d), (2.2e), and (2.2f)]{AT92}$)$
If $p\in L^{\infty}((a,b);dx)$ satisfies \eqref{3.1} for some constants $0<\lambda\leq \Lambda$, then
\begin{equation}
\dom \Big(\big(L_{p,0,0,0}^{(\theta_a,\theta_b)}\big)^{1/2}\Big)=\dom\big(\mathfrak{q}_{p,0,0,0}^{(\theta_a,\theta_b)} \big),\quad \theta_a,\theta_b\in \{0,\pi/2\}.
\end{equation}
\end{theorem}

Our principal new result in the case of Dirichlet or Neumann boundary conditions then reads as follows.

\begin{theorem}\lb{t3.5}
Assume Hypothesis \ref{h3.1}.  Then
\begin{align}\lb{3.56}
\begin{split} 
& \dom \Big(\big(L_{p,q,r,s}^{(\theta_a,\theta_b)}\big)^{1/2}\Big) 
= \dom \Big(\big(\big(L_{p,q,r,s}^{(\theta_a,\theta_b)}\big)^*\big)^{1/2}\Big)    \\ 
& \quad = \dom\big(\mathfrak{q}_{p,q,r,s}^{(\theta_a,\theta_b)} \big) = \dom\big(\mathfrak{q}_{1,0,0,0}^{(\theta_a,\theta_b)} \big), \quad \theta_a,\theta_b\in \{0,\pi/2\}. 
\end{split} 
\end{align} 
\end{theorem}
\begin{proof}
The last equality in \eqref{3.56} follows by definition and the elementary observation in \eqref{3.14}, 
since the domain of $q_p^{(\theta_a,\theta_b)}$ is independent of $p$. The proof of 
$\dom \Big(\big(L_{p,q,r,s}^{(\theta_a,\theta_b)}\big)^{1/2}\Big) 
= \dom\big(\mathfrak{q}_{p,q,r,s}^{(\theta_a,\theta_b)} \big)$ in \eqref{3.56} is essentially identical to the proof of Theorem \ref{t2.5} so we omit technical details, and only provide a sketch of the proof. 
Fixing $\theta_a,\theta_b\in \{0,\pi/2\}$, and mimicking the two-step process used in the proof 
of Theorem \ref{t2.5}, one first proves 
\begin{equation}\lb{3.58}
\dom \Big(\big(L_{p,q,r,0}^{(\theta_a,\theta_b)}\big)^{1/2}\Big)=\dom\big(\mathfrak{q}_{p,q,r,0}^{(\theta_a,\theta_b)} \big).
\end{equation}
In order to do so, note that a resolvent equation analogous to \eqref{2.36} holds for $L_{p,q,r,0}^{(\theta_a,\theta_b)}$ and $L_{p,0,0,0}^{(\theta_a,\theta_b)}$, and it reads
\begin{align}
&\big(L_{p,q,r,0}^{(\theta_a,\theta_b)}-zI_{L^2((a,b);dx)}\big)^{-1}=\big(L_{p,0,0,0}^{(\theta_a,\theta_b)}-zI_{L^2((a,b);dx)}\big)^{-1}    \no \\
&\quad-\ol{\big(L_{p,0,0,0}^{(\theta_a,\theta_b)}-zI_{L^2((a,b);dx)}\big)^{-1}B_1^*}\no\\
&\qquad \times\bigg[I_{L^2((a,b);dx)}+\ol{A_1\big(L_{p,0,0,0}^{(\theta_a,\theta_b)}-zI_{L^2((a,b);dx)}\big)^{-1}B_1^*} \bigg]^{-1}\no\\
&\qquad \times A_1\big(L_{p,0,0,0}^{(\theta_a,\theta_b)}-zI_{L^2((a,b);dx)}\big)^{-1},    \lb{3.59} \\
&\hspace*{1.2cm}z\in \bigg\{\zeta\in \rho\big(L_{p,0,0,0}^{(\theta_a,\theta_b)}\big)\, \bigg| \, 1\in \rho\bigg(\ol{A_1\big(L_{p,0,0,0}^{(\theta_a,\theta_b)}-\zeta I_{L^2(\bbR;dx)}\big)^{-1}B_1^*} \bigg)\bigg\},\no
\end{align}
where the operators $A_1$ and $B_1$ are defined by 
\begin{align}
\begin{split}
&A_1:L^2((a,b);dx)\rightarrow L^2((a,b);dx)^2,\quad A_1f=\begin{pmatrix}
f' \\
e^{i\text{Arg}(q)}|q|^{1/2}f
\end{pmatrix},\\
&B_1:L^2((a,b);dx)\rightarrow L^2((a,b);dx)^2,\quad B_1f=\begin{pmatrix}
\ol{r}f \\
|q|^{1/2}f
\end{pmatrix},
\end{split}\lb{3.60}\\
&\hspace*{2.45cm}f\in \dom(A_1)=\dom(B_1)=W^{1,2}((a,b)).\no
\end{align}
The operators $A_1$ and $B_1$ are closed, densely defined, linear operators from the Hilbert space $L^2((a,b);dx)$ to $L^2((a,b);dx)^2$, with
\begin{equation}\lb{3.61}
\dom\big(L_{p,0,0,0}^{(\theta_a,\theta_b)} \big)\subset \dom(A_1) \, \text{ and } 
\, \dom\big(\big(L_{p,0,0,0}^{(\theta_a,\theta_b)} \big)^*\big)\subset \dom(B_1).
\end{equation}
Analogous to \eqref{2.39}, one estimates
\begin{equation}\lb{3.62}
\Big\|A_1\big(L_{p,0,0,0}^{(\theta_a,\theta_b)}+EI_{L^2((a,b);dx)} \big)^{-1/2} \Big\|_{\cB(L^2((a,b);dx),L^2((a,b);dx)^2)}\leq C_0, \quad E>E_0,
\end{equation}
for constants $E_0>0$ and $C_0>0$, by applying Lemma \ref{2.4} and Theorem \ref{t3.3a}.  In addition, utilizing \eqref{3.10}, one can mimic the proof of Theorem \ref{ta.1} to prove a statement analogous to \eqref{a.1} which states
\begin{equation}\lb{3.63}
\begin{split}
\Big\| \phi\big(L_{1,0,0,0}^{(\theta_a,\theta_b)}+EI_{L^2((a,b);dx)}\big)^{-1/2} \Big\|_{\cB(L^2(\bbR;dx))}\leq C_1E^{-1/4},&\\
E>E_1,\, \phi\in\big\{|r|,|s|,|q|^{1/2}\big\},&
\end{split}
\end{equation}
for constants $C_1>0$, and $E_1\geq 1$.  Consequently, \eqref{3.63} implies
\begin{align}
&\Big\|B_1\big(\big(L_{p,0,0,0}^{(\theta_a,\theta_b)}\big)^*+EI_{L^2((a,b);dx)}\big)^{-1/2}f\Big\|_{L^2((a,b);dx)^2}\no\\
& \quad \leq \Big\{\Big\||r|\big(L_{1,0,0,0}^{(\theta_a,\theta_b)}+EI_{L^2((a,b);dx)}\big)^{-1/2} \Big\|_{\cB(L^2((a,b);dx))} \no\\
&\qquad+\Big\||q|^{1/2}\big(L_{1,0,0,0}^{(\theta_a,\theta_b)}+EI_{L^2((a,b);dx)}\big)^{-1/2} \Big\|_{\cB(L^2((a,b);dx))}\Big\}\no\\
&\qquad \quad  \times M_1\|f\|_{L^2((a,b);dx)}\no\\
& \quad \leq C_2E^{-1/4}\|f\|_{L^2((a,b);dx)},\quad f\in L^2((a,b);dx),\, E>E_1,\lb{3.64}
\end{align}
for $E$-independent constants $M_1>0$ and $C_2>0$.  In \eqref{3.64}, we have used Lemma \ref{l2.4} and Theorem \ref{t3.3a}, noting that 
\begin{equation}\lb{3.65}
\dom \Big(\big(\big(L_{p,0,0,0}^{(\theta_a,\theta_b)}\big)^*\big)^{1/2}\Big) = \dom \Big(\big(L_{1,0,0,0}^{(\theta_a,\theta_b)}\big)^{1/2}\Big).
\end{equation}
Consequently, applying \eqref{3.62} and \eqref{3.64}, one obtains the estimate,
\begin{align}
&\Big\|\ol{A_1\big(L_{p,0,0,0}^{(\theta_a,\theta_b)}+EI_{L^2((a,b);dx)}\big)^{-1}B_1^*}\Big\|_{\cB(L^2((a,b);dx)^2)}\no\\
&\quad \leq \Big\|A_1\big(L_{p,0,0,0}^{(\theta_a,\theta_b)}+EI_{L^2((a,b);dx)}\big)^{-1/2}\Big\|_{\cB(L^2((a,b);dx),L^2((a,b);dx)^2)}\no\\
&\quad \quad \times \Big\|\ol{\big(L_{p,0,0,0}^{(\theta_a,\theta_b)}+EI_{L^2((a,b);dx)}\big)^{-1/2}B_1^*}\Big\|_{\cB(L^2((a,b);dx)^2,L^2((a,b);dx))}\no\\
&\quad \leq C_0C_2E^{-1/4}, \quad E>\max\{E_0,E_1\},\lb{3.66}
\end{align}
which proves
\begin{equation}\lb{3.67}
1\in \rho\Big(\ol{A_1\big(L_{p,0,0,0}^{(\theta_a,\theta_b)}+EI_{L^2((a,b);dx)}\big)^{-1}B_1^*} \Big), 
\end{equation}
for $E > 0$ sufficiently large. Subsequently, since (cf. Theorem \ref{t3.3a} and \eqref{3.60})
\begin{equation}\lb{3.68}
\dom \Big(\big(L_{p,0,0,0}^{(\theta_a,\theta_b)} \big)^{1/2}\Big) \subseteq \dom(A_1), \quad \dom \Big(\big(\big(L_{p,0,0,0}^{(\theta_a,\theta_b)} \big)^*\big)^{1/2}\Big) \subseteq \dom(B_1),
\end{equation}
we may apply Corollary \ref{cd.5a}, yielding
\begin{equation}\lb{3.69}
\dom \Big(\big(L_{p,q,r,0}^{(\theta_a,\theta_b)}\big)^{1/2}\Big)=\dom \Big(\big(L_{p,0,0,0}^{(\theta_a,\theta_b)}\big)^{1/2}\Big)=\dom\big(\mathfrak{q}_{p,0,0,0}^{(\theta_a,\theta_b)} \big),
\end{equation}
and subsequently the equality in \eqref{3.58}.

In the final stage of the proof, we proceed to prove the general statement, that is, 
\begin{equation}
\dom \Big(\big(L_{p,q,r,s}^{(\theta_a,\theta_b)}\big)^{1/2}\Big) 
= \dom\big(\mathfrak{q}_{p,q,r,s}^{(\theta_a,\theta_b)} \big),  \quad 
\theta_a,\theta_b \in \{0, \pi/2\}.     \lb{3.69a} 
\end{equation} 
Having shown \eqref{3.58}, we again apply Corollary \ref{cd.5a} to show 
\begin{equation}\lb{3.70}
\dom \Big(\big(L_{p,q,r,s}^{(\theta_a,\theta_b)}\big)^{1/2}\Big) = \dom \Big(\big(L_{p,q,r,0}^{(\theta_a,\theta_b)}\big)^{1/2}\Big).
\end{equation}
Then, \eqref{3.69a} follows from \eqref{3.69} and the fact that 
\begin{equation}\lb{3.71}
\dom\big(\mathfrak{q}_{p,0,0,0}^{(\theta_a,\theta_b)} \big) = \dom\big(\mathfrak{q}_{p,q,r,s}^{(\theta_a,\theta_b)} \big).
\end{equation}
One notes that
\begin{align}
&\big(L_{p,q,r,s}^{(\theta_a,\theta_b)}-zI_{L^2((a,b);dx)}\big)^{-1} 
= \big(L_{p,q,r,0}^{(\theta_a,\theta_b)}-zI_{L^2((a,b);dx)}\big)^{-1}    \no \\
&\quad-\ol{\big(L_{p,q,r,0}^{(\theta_a,\theta_b)}-zI_{L^2((a,b);dx)}\big)^{-1}B_2^*}\no\\
&\qquad \times\bigg[I_{L^2((a,b);dx)}+\ol{A_2\big(L_{p,q,r,0}^{(\theta_a,\theta_b)}-zI_{L^2((a,b);dx)}\big)^{-1}B_2^*} \bigg]^{-1}\no\\
&\qquad \times A_2\big(L_{p,q,r,0}^{(\theta_a,\theta_b)}-zI_{L^2((a,b);dx)}\big)^{-1},    \lb{3.72} \\
&\hspace*{1.2cm}z\in \Big\{\zeta\in \rho\big(L_{p,q,r,0}^{(\theta_a,\theta_b)}\big)\, \Big| \, 1\in \rho\Big(\ol{A_2\big(L_{p,q,r,0}^{(\theta_a,\theta_b)}-\zeta I_{L^2((a,b);dx)}\big)^{-1}B_2^*} \Big)\Big\},\no
\end{align}
where the operators $A_2,B_2:L^2((a,b);dx)\rightarrow L^2((a,b);dx)$ are defined by 
\begin{equation}\lb{3.73}
A_2f=-sf, \quad B_2f=f',\quad
\dom(A_2)=\dom(B_2)=W^{1,2}((a,b)).
\end{equation}
Indeed, one can again prove a decay estimate for $A_2$, $B_2$, and $L_{p,q,r,0}^{(\theta_a,\theta_b)}$ of the type
\begin{align}
&\Big\|\ol{A_2\big(L_{p,q,r,0}^{(\theta_a,\theta_b)}+EI_{L^2((a,b);dx)}\big)^{-1}B_2^*}\Big\|_{\cB(L^2((a,b);dx))}\no\\
&\quad \leq \Big\|A_2\big(L_{p,q,r,0}^{(\theta_a,\theta_b)}+EI_{L^2((a,b);dx)}\big)^{-1/2}\Big\|_{\cB(L^2((a,b);dx))}\no\\
&\quad \quad \times \Big\|\ol{\big(L_{p,q,r,0}^{(\theta_a,\theta_b)}+EI_{L^2((a,b);dx)}\big)^{-1/2}B_2^*}\Big\|_{\cB(,L^2((a,b);dx))}\no\\
&\quad \leq CE^{-\alpha}, \quad E>E_0,
\end{align}
for appropriate constants $C>0$, $\alpha>0$, and $E_0>0$ to show that
\begin{equation}\lb{3.74}
1\in \rho\Big(\ol{A_2\big(L_{p,q,r,0}^{(\theta_a,\theta_b)}+EI_{L^2((a,b);dx)}\big)^{-1}B_2^*} \Big), 
\end{equation}
for $E > 0$ sufficiently large. As a result, the right-hand side of \eqref{3.72} defines the resolvent of a densely defined, closed, linear operator by \cite[Theorem 2.3]{GLMZ05}.  The operator so defined coincides with $L_{p,q,r,s}^{(\theta_a,\theta_b)}$ by Theorem \ref{te.3}.  Since
\begin{equation}\lb{3.75}
\dom \Big(\big(L_{p,q,r,0}^{(\theta_a,\theta_b)} \big)^{1/2}\Big) \subseteq \dom(A_2), \quad \dom \Big(\big(\big(L_{p,q,r,0}^{(\theta_a,\theta_b)} \big)^*\big)^{1/2}\Big) \subseteq \dom(B_2),
\end{equation}
another application of Corollary \ref{cd.5a} yields the desired result, completing the proof.
\end{proof}

Our next result complements Theorem \ref{t3.5}. In it, we study the particular case $p=1$ a.e. in 
$(a,b)$, but loosen the restrictions on the boundary condition parameters and now allow for arbitrary 
separated boundary conditions, that is, $\theta_a,\theta_b\in \bbS_{\pi}$.  The key component of the 
proof is a Krein-type resolvent identity which reduces the resolvent of $L_{1,0,0,0}^{(\theta_a,\theta_b)}$ 
to a finite rank perturbation of the resolvent of the self-adjoint operator $L_{1,0,0,0}^{(0,0)}$ for which 
equality of square root domains holds trivially. 

\begin{lemma} \lb{l3.3}
The identity 
\begin{equation}\lb{3.15}
\dom \Big(\big(L_{1,0,0,0}^{(\theta_a,\theta_b)}\big)^{1/2}\Big)=\dom \Big(\big(\big(L_{1,0,0,0}^{(\theta_a,\theta_b)}\big)^*\big)^{1/2}\Big)=\dom\big(\mathfrak{q}_{1,0,0,0}^{(\theta_a,\theta_b)} \big), 
\quad \theta_a,\theta_b\in \bbS_{\pi},
\end{equation}
holds.
\end{lemma}
\begin{proof}
In order to establish \eqref{3.15}, it suffices to prove that
\begin{equation}\lb{3.16}
\dom \Big(\big(L_{1,0,0,0}^{(\theta_a,\theta_b)}\big)^{1/2}\Big)\subseteq \dom\big(\mathfrak{q}_{1,0,0,0}^{(\theta_a,\theta_b)} \big),\quad  \theta_a,\theta_b\in \bbS_{\pi}.
\end{equation}
Indeed, if \eqref{3.16} holds for all $\theta_a, \theta_b \in \bbS_{\pi}$, one infers for any given pair $\theta_a$, $\theta_b$ that
\begin{align}
\dom \Big(\big(\big(L_{1,0,0,0}^{(\theta_a,\theta_b)}\big)^*\big)^{1/2}\Big)\subseteq \dom\big(\mathfrak{q}_{1,0,0,0}^{(\ol{\theta_a},\ol{\theta_b})} \big) = \dom\big(\mathfrak{q}_{1,0,0,0}^{(\theta_a,\theta_b)} \big),\lb{3.17a}
\end{align}
directly making use of \eqref{3.16} and the facts,
\begin{align}
& \big(L_{1,0,0,0}^{(\theta_a,\theta_b)}\big)^* = L_{1,0,0,0}^{(\ol{\theta_a},\ol{\theta_b})}, 
\quad  \theta_a,\theta_b\in \bbS_{\pi},    \lb{3.17aa}\\
& \dom\big(\mathfrak{q}_{1,0,0,0}^{(\ol{\theta_a},\ol{\theta_b})} \big) 
= \dom\big(\mathfrak{q}_{1,0,0,0}^{(\theta_a,\theta_b)} \big),\quad  \theta_a,\theta_b\in \bbS_{\pi}. \lb{3.17}
\end{align}
We note that \eqref{3.17aa} follows from the 1st representation theorem \cite[Theorem VI.2.1]{Ka80} in conjunction with \cite[Theorem VI.2.5]{Ka80} and \eqref{3.3}--\eqref{3.6}, while \eqref{3.17} is a 
consequence of the explicit representations of the form domains in \eqref{3.3}--\eqref{3.6}.  The full 
string of domain equalities in \eqref{3.15} is an immediate application of \cite[Corollary to Theorem 1]{Ka62}.

The proof of \eqref{3.16} employs three basic facts.  First, by \cite[Proposition 3.1.9 (a)]{Ha06},
\begin{equation}\lb{3.18}
\begin{split}
\dom \Big(\big(L_{1,0,0,0}^{(\theta_a,\theta_b)}\big)^{1/2}\Big)=\ran \Big(\big(L_{1,0,0,0}^{(\theta_a,\theta_b)}+EI_{L^2((a,b);dx)}\big)^{-1/2}\Big),&\\
E>E(\theta_a,\theta_b),\; \theta_a,\theta_b\in \bbS_{\pi},&
\end{split}
\end{equation}
for an appropriate constant $E(\theta_a,\theta_b)>0$.  Second, the resolvent operator 
\begin{equation}\lb{3.19}
\big(L_{1,0,0,0}^{(\theta_a,\theta_b)}-zI_{L^2((a,b);dx)}\big)^{-1},\quad  z\in \rho \big(L_{1,0,0,0}^{(\theta_a,\theta_b)}\big),
\end{equation}
is an integral operator with a (semi-separable) integral kernel (cf., e.g., \cite{GM03}) which we denote by 
\begin{equation}\lb{3.20}
G_{1,0,0,0}^{(\theta_a,\theta_b)}\big(z,x,x'\big),\quad x,x'\in (a,b),\; z\in \rho \big(L_{1,0,0,0}^{(\theta_a,\theta_b)}\big).
\end{equation}
Third, if $S$ is positive-type operator and $H(t,x,x^{\p})$ is an integral kernel such that
\begin{equation}\lb{3.21}
\Big[\big(S+t I_{L^2(J;dx)}\big)^{-1}u\Big](x)=\int_{J} dx' \, H(t,x,x^{\p})u(x^{\p}), \quad x\in J, \; t>0,  
\end{equation}
$J \subset \bbR$ an appropriate interval, then the operator $S^{-q}$, 
$0 < q < 1$, has the integral kernel $R^q (\cdot,\cdot)$ 
(cf., e.g., \cite[Sect.\ 16]{KZPS76}) 
\begin{equation}\lb{3.22}
R^q (x,x^{\p})=\frac{\sin(\pi q)}{\pi} \int_{0}^{\infty} dt \, 
t^{-q} H(t,x,x^{\p}), \quad x, x' \in J.  
\end{equation}

At this point, one notes that the m-sectorial operator associated with $\mathfrak{q}_{1,0,0,0}^{(\theta_a,\theta_b)}$ is given by
\begin{align}
\begin{split} 
&L_{1,0,0,0}^{(\theta_a,\theta_b)}f=-f'',\quad f\in \dom(L_{1,0,0,0}^{(\theta_a,\theta_b)})=\{g\in AC([a,b])\, |\, \gamma_{\theta_a,\theta_b}(g)=0,     \\
&\hspace*{5.45cm} g'\in AC([a,b]),\, g''\in L^2((a,b);dx)\},     \lb{r3.43}
\end{split} 
\end{align}
where $\gamma_{\theta_a,\theta_b}$ denotes the boundary trace map associated with the boundary $\{a,b\}$ of $(a,b)$ and the parameters $\theta_a,\theta_b\in \bbS_{\pi}$,
\begin{align}
\gamma_{\theta_a,\theta_b}:\left\{\begin{array}{c}
C^1([a,b])\rightarrow \bbC^2,\\
g\mapsto
\begin{pmatrix}
\cos(\theta_a)g(a)+\sin(\theta_a)g'(a) \\
\cos(\theta_b)g(b)-\sin(\theta_b)g'(b) 
\end{pmatrix}
\end{array}\right., \quad \theta_a,\theta_b\in \bbS_{\pi}.\lb{r3.44}
\end{align}
For notational convenience, we introduce the differential expression $\tau$ as follows
\begin{align}
\tau f = -f'',\quad f\in \mathfrak{D}_{\tau}=\{f\in AC([a,b])\,| \, f'\in AC([a,b])\}.\lb{r3.46}
\end{align}
Comparing with \eqref{3.3}--\eqref{3.6}, one needs to distinguish the four cases: $(i)$ $\theta_a\in \bbS\backslash \{0\}$, $\theta_b=0$; $(ii)$ $\theta_b\in \bbS\backslash \{0\}$, $\theta_a=0$; $(iii)$ $\theta_a,\theta_b\in \bbS\backslash\{0\}$; $(iv)$ $\theta_a=\theta_b=0$.  In case $(iv)$, the underlying operator is self-adjoint, and equality of square root domains holds trivially.  Here we only consider the details of case $(i)$ as the other cases are handled similarly.  To this end, let $\theta_a\in \bbS_{\pi}\backslash\{0\}$ be fixed and set $\theta_b=0$.

For each $z\in \rho(L_{1,0,0,0}^{(0,0)})$, let $u_j(z,\, \cdot\,)$, $j=1,2$, denote solutions of $(\tau - z)u=0$ satisfying
\begin{align}
\begin{split}
u_1(z,a)=0,\quad u_1(z,b)=1,&\\
u_2(z,a)=1,\quad u_2(z,b)=0,&
\end{split}
\quad z\in \rho(L_{1,0,0,0}^{(0,0)}).\lb{r3.47}
\end{align}
Then one infers that 
\begin{align}
d_{\theta_a,0}(z):=\cot(\theta_a)+u_2^{[1]}(z,a)\neq 0,\quad z\in \rho(L_{1,0,0,0}^{(0,0)})\cap \rho(L_{1,0,0,0}^{(\theta_a,0)}).\lb{r3.48}
\end{align}
Indeed, if $z\in \rho(L_{1,0,0,0}^{(0,0)})\cap \rho(L_{1,0,0,0}^{(\theta_a,0)})$ and $d_{\theta_a,0}(z)=0$, then
\begin{align}
0&= \sin(\theta_a)d_{\theta_a,0}(z)\no\\
&= \cos(\theta_a)+\sin(\theta_a)u_2^{[1]}(z,a)\no\\
&=\cos(\theta_a)u_2(z,a)+\sin(\theta_a)u_2^{[1]}(z,a).\lb{r3.49}
\end{align}
In addition, 
\begin{align}
\cos(0)u_2(z,b)-\sin(0)u_2^{[1]}(z,b)=0,\lb{r3.50}
\end{align}
and it follows that $z$ is an eigenvalue of $L_{1,0,0,0}^{(\theta_a,0)}$ with $u_2(z,\, \cdot\,)$ as a corresponding eigenfunction.  This, however, is an obvious contradiction.

The utility of $d_{\theta_a,0}(z)$ and $u_2(z,\, \cdot\,)$ is that they allow to express the resolvent of $L_{1,0,0,0}^{(\theta_a,0)}$ in terms of the resolvent of $L_{1,0,0,0}^{(0,0)}$ using a Krein-type resolvent formula,
\begin{align}
& \big(L_{1,0,0,0}^{(\theta_a,0)} -zI_{L^2((a,b);dx)} \big)^{-1} = \big(L_{1,0,0,0}^{(0,0)}-zI_{L^2((a,b);dx)} \big)^{-1}      \lb{r3.51} \\
&\quad -d_{\theta_a,0}(z)^{-1}(u_2(\ol{z},\,\cdot\,),\, \cdot\,)_{L^2((a,b);dx)}u_2(z,\, \cdot\,),  
\quad z\in \rho(L_{1,0,0,0}^{(0,0)})\cap \rho(L_{1,0,0,0}^{(\theta_a,0)}).    \no 
\end{align}
Krein-type formulas of this type were derived in \cite[Theorem~3.1(ii)]{CGNZ12} for self-adjoint Sturm--Liouville operators.  However, self-adjointness is inessential; the proofs of the Krein-type formula presented in \cite{CGNZ12} extend to the present (generally) non-self-adjoint situation without modification.  The function $u_2(z,\,\cdot\,)$ may be computed explicitly
\begin{align}
u_2(z,x)=\frac{\sin(z^{1/2}(b-x))}{\sin(z^{1/2}(b-a))},\quad z\in \rho(L_{1,0,0,0}^{(0,0)})\cap \rho(L_{1,0,0,0}^{(\theta_a,0)}).\lb{3.54a}
\end{align}

Moreover, \eqref{r3.51} immediately yields an identity for the integral kernel of the resolvent (i.e., the Green's function), $G_{1,0,0,0}^{(\theta_a,0)}(z,\dott,\dott)$, of $L_{1,0,0,0}^{(\theta_a,0)}$,
\begin{align}
\begin{split}
G_{1,0,0,0}^{(\theta_a,0)}(z,x,x')=G_{1,0,0,0}^{(0,0)}(z,x,x')-d_{\theta_a,0}(z)^{-1}u_2(\ol{z},x)u_2(z,x'),&\\
x,x'\in [a,b],\, z\in \rho(L_{1,0,0,0}^{(0,0)})\cap \rho(L_{1,0,0,0}^{(\theta_a,0)}).&
\end{split}\lb{r3.52}
\end{align}
As both $L_{1,0,0,0}^{(\theta_a,0)}$ and $L_{1,0,0,0}^{(0,0)}$ are m-sectorial, there exists an $E(\theta_a)>0$ such that
\begin{align}
-E\in \rho(L_{1,0,0,0}^{(0,0)})\cap \rho(L_{1,0,0,0}^{(\theta_a,0)}),\quad E>E(\theta_a).\lb{r3.53}
\end{align}
Since
\begin{align}
\begin{split}
\Big[\big(L_{1,0,0,0}^{(\theta_a,0)}-zI_{L^2((a,b);dx)} \big)^{-1}f\Big](x) 
= \int_a^b dx'\, G_{1,0,0,0}^{(\theta_a,0)}(-E,x,x')f(x')&\\
\text{a.e.}\, x\in (a,b),\, f\in L^2((a.b);dx),\, E>E(\theta_a),&
\end{split}\lb{r3.54}
\end{align}
one may apply \eqref{3.22} to compute an integral kernel for the square root operator,
\begin{align}
\big(L_{1,0,0,0}^{(\theta_a,0)}-zI_{L^2((a,b);dx)} \big)^{-1/2},\quad E>E(\theta_a).\lb{r3.55}
\end{align}
Indeed, according to \eqref{3.22} and \eqref{r3.54}, the desired integral kernel is
\begin{align}
&R_{1,0,0,0}^{(\theta_a,0)}(-E,x,x')  
=\frac{1}{\pi}\int_0^{\infty} dt\, t^{-1/2}\Big[ G_{1,0,0,0}^{(0,0)}(-(t+E),x,x') -T(t+E,x,x')\Big] \no\\
& \quad =R_{1,0,0,0}^{(0,0)}(-E,x,x')-\frac{1}{\pi}\int_0^{\infty}dt\, t^{-1/2}T(t+E,x,x'),\quad x,x'\in [a,b],\, E>E(\theta_a),\lb{r3.56}   
\end{align}
where
\begin{align}
T(t,x,x') = \frac{\sinh(t^{1/2}(b-x))\sinh(t^{1/2}(b-x'))}{d_{\theta_a,0}(-t)[\sinh(t^{1/2}(b-a))]^{{}^2}},\quad x,x'\in (a,b),\,  t>E(\theta_a).\lb{4.61.0}
\end{align}
We note that $T(t,x,x')$ satisfies the following estimate
\begin{align}
|T(t,x,x')|&\leq Ct^{-1/2}\big\{e^{-t^{1/2}(x+x'-2a)}+e^{-t^{1/2}(2b+x-x'-2a)}+e^{-t^{1/2}(2b+x'-x-2a)}\no\\
&\quad +e^{-t^{1/2}(4b-x-x'-2a)} \big\}, \quad t>t_0,\, x,x'\in (a,b),\lb{4.62.0}
\end{align}
for some constants $C>0$ and $t>0$, implying that the second integral in \eqref{r3.56} is bounded by
\begin{align}
&\bigg|\int_0^{\infty}dt\, t^{-1/2}T(t+E,x,x')\bigg|\no\\
&\quad \leq 2C\big\{K_0(E^{1/2}(x+x'-2a)) + K_0(E^{1/2}(2b+x-x'-2a)) \no\\
&\qquad + K_0(E^{1/2}(2b+x'-x-2a)) + K_0(E^{1/2}(4b-x-x'-2a)) \big\},\lb{3.62aa}\\
&\hspace*{5.05cm}x,x'\in (a,b),\, E>\max\{E(\theta_a),t_0\}.\no
\end{align}
In \eqref{3.62aa}, $K_0(\,\cdot\,)$ denotes the zeroth order modified Bessel function, and we have used its integral representation (cf., e.g., \cite[Sect. 9.6]{AS72} and \cite[3.387.6]{GR80}),
\begin{equation}\lb{3.63aa}
\int_{\alpha}^{\infty} (s^2-\alpha^2)^{-1/2}e^{-sy}ds=K_0(\alpha y), \quad \alpha>0,\; y>0.
\end{equation}

Now, in order to prove \eqref{3.16}, fix $E>\max\{E(\theta_a),t_0\}$, and let
\begin{align}
g\in \dom\Big(\big(L_{1,0,0,0}^{(\theta_a,0)} \big)^{1/2} \Big) = \ran\Big(\big(L_{1,0,0,0}^{(\theta_a,0)}-zI_{L^2((a,b);dx)} \big)^{-1/2} \Big),\lb{r3.57}
\end{align}
and suppose that $f\in L^2((a,b);dx)$ such that
\begin{align}
g = \big(L_{1,0,0,0}^{(\theta_a,0)}-EI_{L^2((a,b);dx)} \big)^{-1/2}f.\lb{r3.58}
\end{align}
We will show that $g$ satisfies the requirements to belong to $\dom\big(\mathfrak{q}_{1,0,0,0}^{(\theta_a,0)}\big)$.  

According to \eqref{r3.56} and \eqref{r3.58}, $g$ has the represenation
\begin{align}
&g(x)= \int_a^bdx'\, R_{p,0,0,0}^{(\theta_a,0)}(-E,x,x')\no\\
&= \Big[\big(L_{p,0,0,0}^{(0,0)}-EI_{L^2((a,b);dx)} \big)^{-1/2}f \Big](x)\lb{r3.59}\\
&\quad -\frac{1}{\pi}\int_a^bdx'\int_0^{\infty}dt\, t^{-1/2}T(-(t+E),x,x')f(x')\ \text{a.e.}\, x\in (a,b).\no
\end{align}

The operator $L_{1,0,0,0}^{(0,0)}$ is self-adjoint, so 
\begin{align}
\big(L_{p,0,0,0}^{(0,0)}-EI_{L^2((a,b);dx)} \big)^{-1/2}f \in \dom\big(L_{1,0,0,0}^{(0,0)} \big) = W^{1,2}_0((a,b))\subset\dom\big(\mathfrak{q}_{1,0,0,0}^{(\theta_a,0)}\big).\lb{4.68.0}
\end{align} 

Consequently, the first term on the right-hand side in \eqref{r3.59} (i.e., the term preceding the double integral) belongs to the form domain of $L_{1,0,0,0}^{(0,0)}$, hence it belongs to the form domain of $L_{1,0,0,0}^{(\theta_a,0)}$.  Thus, it suffices to show the second term in \eqref{r3.59} belongs to the form domain of $L_{1,0,0,0}^{(\theta_a,0)}$.  An application of dominated convergence shows that this term is absolutely continuous and its derivative belongs to $L^2((a,b);dx)$.  Finally, one computes
\begin{align}
\begin{split} 
g(b)&= \Big[\big(L_{p,0,0,0}^{(0,0)}-EI_{L^2((a,b);dx)} \big)^{-1/2}f \Big](b)\lb{3.68a}\\
&\quad -\frac{1}{\pi}\int_a^bdx'\int_0^{\infty}dt\, t^{-1/2}T(-(t+E),b,x')f(x')=0,   
\end{split} 
\end{align}
The first term on the right-hand side of \eqref{3.68a} vanishes since functions belonging to the form domain of $L_{1,0,0,0}^{(0,0)}$ must vanish at $x=b$, and the second term vanishes via \eqref{4.61.0}.
\end{proof}

Our next result states that the square root domain of $L_{1,q,r,s}^{(\theta_a,\theta_b)}$ is actually independent of $q$, $r$, $s$, and to a certain extent, the boundary condition parameters $\theta_a$ and $\theta_b$, depending upon whether one, or both, of $\theta_a$ and $\theta_b$ is zero.

We note the form domain in the far right-hand side of \eqref{3.11} only depends on whether $\theta_a$ (or $\theta_b$) is zero or not (cf. eqs.~\eqref{3.3}--\eqref{3.6}).

\begin{theorem} \lb{t3.6} 
Assume Hypothesis \ref{h3.1} with $p=1$ a.e.~in $(a,b)$.  Then
\begin{align}
\dom \Big(\big(L_{1,q,r,s}^{(\theta_a,\theta_b)}\big)^{1/2}\Big) = \dom\big(\mathfrak{q}_{1,q,r,s}^{(\theta_a,\theta_b)} \big) = \dom\big(\mathfrak{q}_{1,0,0,0}^{(\theta_a,\theta_b)} \big), \quad \theta_a,\theta_b\in \bbS_{\pi}.\lb{4.69a}
\end{align}
\end{theorem}
\begin{proof}
The result follows from Lemma \ref{l3.3} in a manner analogous to the way Theorem \ref{t3.5} follows from Theorem \ref{t3.3a}. To prove this result only requires to suitably modify the proof of Theorem \ref{t3.5}: one replaces $p$ everywhere by $1$, lifts the restriction $\theta_a,\theta_b\in \{0,\pi/2\}$, and replaces references to Theorem \ref{t3.3a} by references to Lemma \ref{l3.3}, throughout the proof of Theorem \ref{t3.5}.
\end{proof}

It is possible to extend our approach based on Krein's resolvent formula to the case where 
$ 0 < p^{-1} \in L^1((a,b); dx)$, $p, p' \in AC([a,b])$, using a Green--Liouville-type transformation 
as described, for instance, in \cite[p.\ 186]{Te09}. Since this necessitates a certain degree of smoothness 
for the coefficient $p$ we decided not to pursue this further at this point.

\appendix

\section{Stability of Square Root Domains}  \lb{sA}

The goal of this appendix is to recall the main abstract results on stability of square root domains for operators of the form $T_0+B^*A$ presented in \cite{GHN13}.  We begin with the following basic set of assumptions 
(with $\cH$ and $\cK$ complex, separable Hilbert spaces).

\begin{hypothesis}\lb{hd.1}
$(i)$  Suppose that $T_0$ is a densely defined, closed, linear operator in $\cH$ with nonempty resolvent set,
\begin{equation}\lb{d.1}
\rho(T_0)\neq \emptyset,
\end{equation}
$A:\dom(A)\rightarrow \cK$, $\dom(A)\subseteq \cH$, is a densely defined, closed, linear operator from $\cH$ to $\cK$, and $B:\dom(B)\rightarrow \cK$, $\dom(B)\subseteq \cH$, is a densely defined, closed, linear operator from $\cH$ to $\cK$ such that
\begin{equation}\lb{d.2}
\dom(A)\supseteq \dom(T_0), \quad \dom(B)\supseteq \dom(T_0^{\ast}).
\end{equation}
$(ii)$ For some $($and hence for all\,$)$ $z\in \rho(T_0)$, the operator $-A(T_0-zI_{\cH})^{-1}B^{\ast}$, defined on $\dom(B^{\ast})$, has a bounded extension in $\cK$, denoted by $K(z)$,
\begin{equation}\lb{d.3}
K(z)=-\overline{A(T_0-zI_{\cH})^{-1}B^*} \in \cB(\cK).
\end{equation}
$(iii)$  $1\in \rho(K(z_0))$ for some $z_0\in \rho(T_0)$. 
\end{hypothesis}

Suppose Hypothesis \ref{hd.1} holds and define
\begin{equation}\lb{d.4}
\begin{split}
R(z)=(T_0-zI_{\cH})^{-1}-\overline{(T_0-zI_{\cH})^{-1}B^*}[I_{\cK}-K(z)]^{-1}A(T_0-zI_{\cH})^{-1},&\\
z\in \{\zeta\in \rho(T_0)\,|\,1\in \rho(K(\zeta))\}.&
\end{split}
\end{equation}
Under the assumptions of Hypothesis \ref{hd.1}, $R(z)$ given by \eqref{d.4} defines a densely defined, closed, linear operator $T$ in $\cH$ (cf. \cite{GLMZ05}, \cite{Ka66}) by
\begin{equation}\lb{d.5}
R(z)=(T-zI_{\cH})^{-1}, \quad z\in \{\zeta\in \rho(T_0) \, | \, 1\in \rho(K(\zeta))\}.
\end{equation}
Combining \eqref{d.4} and \eqref{d.5} yields Kato's resolvent equation (in the slightly more general form 
of \cite{GLMZ05}, in which $T_0$ is no longer assumed to be self-adjoint), 
\begin{align}\lb{d.6}
& (T-zI_{\cH})^{-1}=(T_0-zI_{\cH})^{-1}-\overline{(T_0-zI_{\cH})^{-1}B^*}
[I_{\cK}-K(z)]^{-1}A(T_0-zI_{\cH})^{-1},    \no \\
& \hspace*{6.6cm} z\in \{\zeta\in \rho(T_0)\,|\, 1\in \rho(K(\zeta))\}. 
\end{align}
The operator $T$ defined by \eqref{d.5} is an extension of $(T_0+B^*A)|_{\dom(T_0)\cap \dom(B^{\ast}A)}$,
\begin{equation}\lb{d.7}
T\supseteq (T_0+B^{\ast}A)|_{\dom(T_0)\cap \dom(B^{\ast}A)}. 
\end{equation}

We add that the operator sum $(T_0+B^*A)|_{\dom(T_0)\cap \dom(B^*A)}$ can be problematic 
since it is possible that $\dom(T_0)\cap \dom(B^*A)=\{0\}$ (for a pertinent example, see, e.g., 
\cite[Sect.\ I.6]{Si71}; see also \cite{SV85}).  In light of \eqref{d.7}, $T$ defined by \eqref{d.5} 
should be viewed as a generalized sum of $T_0$ and $B^*A$, that is, the perturbation $W$ of 
$T_0$ has been factored according to $W = B^* A$. Under the additional assumption that 
$\dom\big(T_0^{1/2}\big) = \dom\big((T_0^*)^{1/2}\big)$, and the additional hypotheses that \eqref{d.8} below holds, we will note in Theorem \ref{te.3} that $T$ also extends the form sum of $T_0$ and $B^* A$. 

We recall that a linear operator $D$ in $\cH$ is called {\it accretive} if the numerical range of $D$ 
(i.e., the set $\{(f,Df)_{\cH}\in\bbC \,|\, f \in \dom(D), \, \|f\|_{\cH}=1\}$) is a subset of the closed 
right complex half-plane. 
$D$ is called {\it m-accretive} if $D$ is a closed and maximal accretive operator (i.e., $D$ has no proper accretive extension). One recalls that an equivalent definition of an m-accretive operator $D$ in $\cH$ reads
\begin{equation}
(D + \zeta I_{\cH})^{-1} \in \cB(\cH), \quad
\big\|(D + \zeta I_{\cH})^{-1} \big\|_{\cB(\cH)} \leq [\Re(\zeta)]^{-1}, \quad \Re(\zeta) > 0.   \lb{d.7aaa}
\end{equation}
One also recalls that any m-accretive operator is necessarily densely defined.  Moreover,  $D$ is called an {\it m-sectorial} operator with a {\it vertex} $\gamma \in \bbR$ and a corresponding {\it semi-angle} $\theta\in [0,\pi/2)$, in short, {\it m-sectorial}, if $D$ is a maximal 
accretive, closed (and hence densely defined) operator, and the numerical range of $D$ is contained 
in the sector $\cS_{\gamma,\theta}$,
\begin{equation}\lb{d.7aaaa}
\cS_{\gamma,\theta}:=\{\zeta \in \bbC\, |\, |\text{arg}(\zeta-\gamma)|\leq \theta\}, 
\quad \theta \in [0,\pi/2).
\end{equation}

\begin{hypothesis}\lb{hd.3}
$(i)$ Suppose $T_0$ is m-accretive. \\ 
$(ii)$ Suppose that $A:\dom(A)\rightarrow \cK$, $\dom(A)\subseteq \cH$, is a closed, linear operator 
from $\cH$ to $\cK$, and $B:\dom(B)\rightarrow \cK$, $\dom(B)\subseteq \cH$, is a closed, linear 
operator from $\cH$ to $\cK$ such that 
\begin{equation}\lb{d.8}
\dom(A)\supseteq \dom\big(T_0^{1/2}\big), \quad \dom(B)\supseteq \dom\big((T_0^*)^{1/2}\big).
\end{equation}
$(iii)$  Suppose that there exist constants $R>0$ and $E_0>0$ such that
\begin{align}\lb{d.9}
\begin{split} 
\int_R^{\infty}d\lambda\, \lambda^{-1} & \big\|A(T_0+(\lambda+E)I_{\cH})^{-1/2}\big\|_{\cB(\cH,\cK)}   \\
& \times \big\|\overline{(T_0+(\lambda+E)I_{\cH})^{-1/2}B^{\ast}}\big\|_{\cB(\cK,\cH)}<\infty,   
\quad  E \geq E_0,
\end{split} 
\end{align}
with
\begin{align}\lb{d.10}
\begin{split} 
\lim_{E\rightarrow \infty}\int_R^{\infty}d\lambda\, \lambda^{-1} & 
\big\|A(T_0+(\lambda+E)I_{\cH})^{-1/2}\big\|_{\cB(\cH,\cK)}     \\
& \times \big\|\overline{(T_0+(\lambda+E)I_{\cH})^{-1/2}B^{\ast}}\big\|_{\cB(\cK,\cH)}=0,
\end{split} 
\end{align}
and 
\begin{align}
\lim_{E\rightarrow \infty}\big\|A(T_0+EI_{\cH})^{-1/2}\big\|_{\cB(\cH,\cK)}\big\|\overline{(T_0+EI_{\cH})^{-1/2}B^{\ast}}\big\|_{\cB(\cK,\cH)}=0.\lb{2.25ccc}
\end{align}
\end{hypothesis}


The principal stability result of square root domains shown in \cite{GHN13} then reads as follows.

\begin{theorem}\lb{td.4}$($\cite[Theorem~2.4]{GHN13}$)$
Assume Hypotheses \ref{hd.1} and \ref{hd.3}.  If $T$ defined as in \eqref{d.5} is m-accretive, then
\begin{equation}\lb{d.11}
\dom\big(T^{1/2}\big) = \dom\big(T_0^{1/2}\big).
\end{equation}
\end{theorem}

Next, we strengthen Hypothesis \ref{hd.3} as follows.

\begin{hypothesis}\lb{hd.5}
In addition to the assumptions in Hypothesis \ref{hd.3} suppose that 
\begin{equation}\lb{d.12}
\dom\big(T_0^{1/2}\big)=\dom\big((T_0^*)^{1/2}\big).
\end{equation}
\end{hypothesis}

Under the strengthened Hypothesis \ref{hd.5}, one obtains stability of square root domains.

\begin{corollary}\lb{cd.5a}$($\cite[Corollary~2.6]{GHN13}$)$
Assume Hypotheses \ref{hd.1} and \ref{hd.5}.  If $T$ defined as in \eqref{d.5} is m-accretive, then
\begin{equation}\lb{d.13}
\dom\big(T^{1/2}\big) = \dom\big(T_0^{1/2}\big) = \dom\big((T_0^*)^{1/2}\big)= \dom\big((T^*)^{1/2}\big).
\end{equation}
\end{corollary}

\section{Sesquilinear Forms and Associated Operators}  \lb{sB}

In this appendix we briefly recall the precise connection between Kato's 
resolvent equation \eqref{d.4} defining the operator $T$ on one hand, and the form sum of $T_0$ and an extension of $B^*A$ under somewhat stronger hypotheses on $T_0$, $A$, and $B$ than those in 
Hypothesis \ref{hd.1}, on the other hand.

We start with the following basic assumptions (with $\cH$ and $\cK$ complex, separable Hilbert spaces). 

\begin{hypothesis} \lb{he.2}
$(i)$ Suppose $T_0$ is m-sectorial in $\cH$ and 
\begin{equation}\lb{e.7a}
\dom\big(T_0^{1/2}\big)=\dom\big((T_0^*)^{1/2}\big).
\end{equation}
$(ii)$ Assume that $A:\dom(A)\rightarrow \cK$, $\dom(A)\subseteq \cH$, is a closed, linear operator 
from $\cH$ to $\cK$, and $B:\dom(B)\rightarrow \cK$, $\dom(B)\subseteq \cH$, is a closed, linear 
operator from $\cH$ to $\cK$ such that 
\begin{equation}\lb{e.8a}
\dom(A)\supseteq \dom\big(T_0^{1/2}\big), \quad \dom(B)\supseteq \dom\big(T_0^{1/2}\big),
\end{equation}
and that for some $0 \leq a < 1$, $b \geq 0$, 
\begin{align}
\begin{split} 
& \|A f\|_{\cK}^2 \leq a \Re[\gq^{}_{T_0}(f,f)] + b \|f\|_{\cH}^2, \quad f \in \dom\big(T_0^{1/2}\big),    \\
& \|B f\|_{\cK}^2 \leq a \Re[\gq^{}_{T_0}(f,f)] + b \|f\|_{\cH}^2, \quad f \in \dom\big(T_0^{1/2}\big).  
\lb{e.9a}
\end{split} 
\end{align}
$(iii)$ Suppose that 
\begin{equation}
\lim_{E \uparrow \infty} \|K(-E)\|_{\cB(\cK)} = 0,    \lb{e.44} 
\end{equation} 
where
\begin{align}
\begin{split} 
K(z) &= - \overline{A(T_0-zI_{\cH})^{-1}B^*}    \\
&= - \big[A (T_0-zI_{\cH})^{-1/2}\big] 
\Big[\ol{\big[B (T_0^* - \ol{z} I_{\cH})^{-1/2}\big]^*}\Big] \in \cB(\cK),  \quad z \in \rho(T_0).  \lb{e.45a} 
\end{split} 
\end{align}
\end{hypothesis}

In the following we denote by $\gq^{}_{T_0}$ the sectorial form associated with $T_0$ in $\cH$ and note 
that according to Remark \ref{t1.2}, assumption \eqref{e.7a} yields that 
\begin{equation}
\dom\big(T_0^{1/2}\big) = \dom(\gq^{}_{T_0}) = \dom\big((T_0^*)^{1/2}\big).   \lb{e.46a}
\end{equation} 
In addition, introducing the form 
\begin{equation}
\gq^{}_W (f,g) = (Bf, Ag)_{\cK}, \quad f, g \in \dom\big(T_0^{1/2}\big)     \lb{e.47}
\end{equation}
(formally corresponding to the operator $B^* A$), it is clear from assumption \eqref{e.9a} that 
$\gq^{}_W$ is relatively bounded with respect to $\gq^{}_{T_0}$ with bound strictly less than one. 
Thus, we may introduce the sectorial form in $\cH$, 
\begin{align}
\begin{split} 
& \gq^{}_{T_0 +_{\gq} W}(f,g) = \gq^{}_{T_0}(f,g) + \gq^{}_W (f,g), \\
& f,g \in \dom(\gq^{}_{T_0 +_{\gq} W}) = \dom(\gq^{}_{T_0}) = \dom\big(T_0^{1/2}\big),   \lb{e.48}
\end{split} 
\end{align}
and hence denote the m-sectorial operator uniquely associated with $\gq^{}_{T_0 +_{\gq} W}$ 
by $T_0 +_{\gq} W$ in the following.  Here ``$+_{\gq}$'' denotes the (quadratic) form sum of two operators.

The principal result of this appendix, establishing equality between $T$ defined according to 
Kato's method \eqref{d.4}, \eqref{d.5} on one hand, and the form sum $T_0 +_{\gq} W$ on 
the other, a result of interest in its own right, was proved in \cite{GHN13} and reads as follows.

\begin{theorem} \lb{te.3}$($\cite[Theorem~A.3]{GHN13}$)$
Assume Hypothesis \ref{he.2}. Then $T$ defined as in 
\eqref{d.4} and \eqref{d.5} coincides with the form sum of $T_0 +_{\gq} W$ in $\cH$, 
\begin{equation}
T = T_0 +_{\gq} W.     \lb{d.49} 
\end{equation}
\end{theorem}

\section{Operator Norm and Form Bounds}  \lb{sC}

This appendix records several technical estimates used throughout the course of the main presentation. 
For notational consistency throughout this manuscript, we follow the notation set forth in \eqref{2.27} 
in this appendix and, therefore, let $L_{1,0,0,0}$ denote the self-adjoint realization of $(-d^2/dx^2)$ 
with domain $W^{1,2}(\bbR)$ in $L^2(\bbR;dx)$.

\begin{theorem}\lb{ta.1}
Assume Hypothesis \ref{h2.1}.  Then there exist constants $C>0$ and $E_0>0$ such that
\begin{equation}\lb{a.1}
\big\| \phi(L_{1,0,0,0}+EI_{L^2(\bbR;dx)})^{-1/2} \big\|_{\cB(L^2(\bbR;dx))}\leq CE^{-1/4}, 
\quad E>E_0,\; \phi\in\big\{|r|,|s|,|q|^{1/2}\big\}.
\end{equation}
\end{theorem}
\begin{proof}
Let $ \phi\in\big\{|r|,|s|,|q|^{1/2}\big\}$ be fixed for the remainder of this proof.  By \eqref{2.16}--\eqref{2.18}, there exist constants $c_{\phi}>0$ and $\varepsilon_{\phi}>0$ such that
\begin{equation}\lb{a.2}
\|\phi f\|_{L^2(\bbR;dx)}\leq \e \|f'\|_{L^2(\bbR;dx)}+\e^{-1}c_{\phi}\|f\|,\quad f\in W^{1,2}(\bbR),\; 0<\e<\e_{\phi}.
\end{equation}
Choosing $f=(L_{1,0,0,0}+EI_{L^2(\bbR;dx)})^{-1/2}g$, $g\in L^2(\bbR;dx)$, in \eqref{a.2} yields the estimate
\begin{align}
&\big\| \phi(L_{1,0,0,0}+EI_{L^2(\bbR;x)})^{-1/2}g \big\|_{L^2(\bbR;dx)}\no\\
&\quad \leq \e \bigg\|\frac{d}{dx}(L_{1,0,0,0}+EI_{L^2(\bbR;x)})^{-1/2}\bigg\|_{\cB(L^2(\bbR;dx))}\|g\|_{L^2(\bbR;dx)}
\no \\
& \qquad +c_{\phi}\e^{-1}E^{-1/2}\|g\|_{L^2(\bbR;dx)}\no\\
& \quad \leq \big(\e M + c_{\phi}\e^{-1}E^{-1/2}\big)\|g\|_{L^2(\bbR;dx)},\quad g\in L^2(\bbR;dx),
\; 0<\e<\e_{\phi},\; E>0,\lb{a.3}
\end{align}
where $M>0$ is an $E$-independent positive constant which bounds the operator norm appearing in the second line of \eqref{a.3} (the existence of such a constant is guaranteed by Lemma \ref{l2.4}), and we have used the spectral theorem for self-adjoint operators to bound the operator norm of $(L_{1,0,0,0}+EI_{L^2(\bbR;dx)})^{-1/2}$ by $E^{-1/2}$ for any $E>0$.  Choosing $\e=E^{-1/4}$, $E>\e_{\phi}^{-4}$, 
throughout \eqref{a.3} yields the operator norm inequality
\begin{equation}\lb{a.4}
\big\| \phi(L_{1,0,0,0}+EI_{L^2(\bbR;x)})^{-1/2} \big\|_{\cB(L^2(\bbR;dx))}\leq (M+c_{\phi})E^{-1/4}, 
\quad E>\e_{\phi}^{-4}.
\end{equation}
Consequently, \eqref{a.1} follows by choosing
\begin{equation}\lb{a.5}
C=\max \big\{c_{|r|}, c_{|s|}, c_{|q|^{1/2}}\big\} \, \text{ and } \, 
E_0=\max \big\{\e_{|r|}^{-4},\e_{|s|}^{-4},\e_{|q|^{1/2}}^{-4} \big\}.
\end{equation}
\end{proof}

By employing a similar strategy, one can prove the following result for the half-line case $(a,\infty)$ with either a Dirichlet or Neumann boundary condition at $x=a$.  We recall that $L_{1,0,0,0}^{a,D}$ 
(resp., $L_{1,0,0,0}^{a,N}$) denotes the free Laplacian $(-d^2/dx^2)$ on $(a,\infty)$ with a Dirichlet (resp., Neumann) boundary condition at $x=a$ (see the discussion surrounding \eqref{4.16} for complete details).

\begin{theorem}\lb{ta.2}
Assume Hypothesis \ref{h4.1}.  Then there exist constants $C>0$ and $E_0>0$ such that
\begin{equation}\lb{a.6}
\begin{split}
\big\| \phi\big(L_{1,0,0,0}^{a,Y}+EI_{L^2((a,\infty);dx)}\big)^{-1/2} \big\|_{\cB(L^2((a,\infty);dx))}\leq CE^{-1/4},&\\ 
\quad E>E_0,\; \phi\in\big\{|r|,|s|,|q|^{1/2}\big\},\; Y\in \{D,N\}.&
\end{split}
\end{equation}
\end{theorem}

The next result is an abstract result for pairs of m-accretive operators and provides uniform bounds on certain operator norms when the square root domain of one operator contains the square root domain of the other.

\begin{lemma}\lb{l2.4}$($\cite[Lemma 2.11]{GHN13}$)$.
If $S$ and $T$ are m-accretive operators in the Hilbert space $\cH$ with
\begin{equation}\lb{2.30}
\dom\big(T^{1/2}\big)\subseteq \dom\big(S^{1/2}\big),
\end{equation}
then there exists a constant $C>0$ such that
\begin{align}
\sup_{E\geq 1}\big\|S^{1/2}(T+EI_{\cH})^{-1/2} \big\|_{\cB(\cH)}&\leq C,\lb{2.31}\\
\sup_{E\geq 1}\big\|(S+EI_{\cH})^{1/2}(T+EI_{\cH})^{-1/2} \big\|_{\cB(\cH)}& \leq C.\lb{2.32}
\end{align}
\end{lemma}

Next, we turn to a finite interval analogue of the Trudinger-type infinitesimal form bound given by Schechter in \cite[Theorem 2.7.1]{Sc81} (for completeness, we provide a proof). 

\begin{theorem}\lb{tb.1}
Let $a,b\in \bbR$ with $a<b$ and suppose that $w\in L^2((a,b);dx)$. Then
\begin{align}
& |f(x)|^2 \leq \e \|f'\|_{L^2((a,b);dx)}^2 + \big((b-a)^{-1} +\e^{-1} \big)\|f\|_{L^2((a,b);dx)}^2,\lb{b.1}  \\
& \|wf\|_{L^2((a,b);dx)}^2 \leq \e N_w\|f' \|_{L^2((a,b);dx)}^2+\big((b-a)^{-1} + \e^{-1} \big)N_w \|f\|_{L^2((a,b);dx)},\lb{b.2}   \no \\
&\hspace*{5.6cm}f\in W^{1,2}((a,b)),\, x\in [a,b],\, \e>0,
\end{align}
where we abbreviated $N_w=\|w\|_{L^2((a,b);dx)}^2$.
\end{theorem}
\begin{proof}
If $f\in AC([a,b])$ with $f'\in L^2((a,b);dx)$, then
\begin{equation}\lb{b.3}
f(x)^2-f(x')^2=2\int_{x'}^x ds\, f'(s)f(s), \quad a\leq x' \leq x \leq b,
\end{equation}
so that
\begin{equation}\lb{b.4}
\big| f(x)^2-f(x')^2 \big| \leq 2 \int_{x'}^x ds\, |f'(s)f(s)|, \quad a \leq x' \leq x \leq b.
\end{equation}
On the other hand, 
\begin{equation}\lb{b.5}
\e|f'(s)|^2+\e^{-1}|f(s)|^2\geq 2\big| f'(s)f(s)\big|\ \text{for a.e.\ } s\in (a,b), \, \e>0.
\end{equation}
As a result, one obtains the estimate
\begin{align}
\big||f(x)|^2-|f(x')|^2 \big|&\leq \big|f(x)^2-f(x')^2 \big|\no\\
&\leq \int_{x'}^x ds\, \big\{ \e|f'(s)|^2+\e^{-1}|f(s)|^2 \big\}   \no \\
&\leq \int_a^bds\, \big\{ \e|f'(s)|^2+\e^{-1}|f(s)|^2 \big\},    \lb{b.6} \\
&\hspace*{1.8cm} a \leq x' \leq x \leq b,\; \e>0,\no
\end{align}
and subsequently,
\begin{equation}\lb{b.7}
|f(x)|^2\leq |f(x')|^2 + \e\int_a^b ds\, |f'(s)|^2 + \e^{-1}\int_a^b ds\, |f(s)|^2,\quad x,x'\in[a,b],\; \e>0.
\end{equation}
Since $f$ is absolutely continuous on $[a,b]$, $|f|^2$ attains its average value on $[a,b]$.  Let $x_0'\in [a,b]$ such that
\begin{equation}\lb{b.8}
|f(x_0')|^2=(b-a)^{-1}\int_a^bds\, |f(s)|^2.
\end{equation}
Choosing $x'=x_0'$ in \eqref{b.7} yields \eqref{b.1}.  Pre-multiplying both sides of \eqref{b.1} by $w(x)$ and integrating from $a$ to $b$ yields \eqref{b.2}.
\end{proof}

The following elementary result on infinitesimal form boundedness has been used repeatedly in the bulk of this paper (again, we provide the proof for completeness).

\begin{lemma}
Suppose that $\mathfrak{Q}$ is a sectorial sesquilinear form and $\mathfrak{q}_j$, $j\in \{1,2\}$, are sesquilinear forms in $\cH$ $($a complex, separable Hilbert space\,$)$.  If $\mathfrak{q}_j$ are 
infinitesimally bounded with respect to $\mathfrak{Q}$, so that 
$\dom(\mathfrak{Q})\subset \dom\big(\mathfrak{q}_j\big)$, and
\begin{equation}\lb{c.1}
\big|\mathfrak{q}_j(f,f) \big|\leq \e |\mathfrak{Q}(f,f)| + \eta_j(\e)\|f\|_{\cH}^2,\quad f\in \dom(\mathfrak{Q}),
\; 0<\e<\e_j, 
\end{equation}
for some constants $\e_j>0$ and functions $\eta_j:(0,\e_j)\rightarrow (0,\infty)$, $j\in \{1,2\}$, then 
$\mathfrak{q}_2$ is infinitesimally bounded with respect to the sum $\mathfrak{Q}+\mathfrak{q}_1$ 
with domain $\dom(\mathfrak{Q})$ and
\begin{equation}\lb{c.2}
\big|\mathfrak{q}_2(f,f) \big|\leq \e |\mathfrak{Q}(f,f)+\mathfrak{q}_1(f,f)| + \eta_0(\e)\|f\|_{\cH}^2,\quad f\in \dom(\mathfrak{Q}),\; 0<\e<\e_0,
\end{equation}
where 
\begin{equation}\lb{c.3}
\e_0=2\min\{1/2,\e_1,\e_2\},\quad \eta_0(\e)=\eta_1(\e/2)+\eta_2(\e/2),\quad 0<\e<\e_0.
\end{equation}
\end{lemma}
\begin{proof}
Beginning with \eqref{c.1}, we use the triangle inequality to obtain
\begin{equation}\lb{c.4}
\begin{split}
\big|\mathfrak{q}_j(f,f) \big|\leq \e \big| \mathfrak{Q}(f,f)+\mathfrak{q}_1(f,f) \big|+\e\big|\mathfrak{q}_1(f,f) \big|+\eta_j(\e)\|f\|_{\cH}^2,&\\
f\in \dom(\mathfrak{Q}),\; 0<\e<\e_j,\, j\in\{1,2\}.&
\end{split}
\end{equation}
Taking $j=1$ in \eqref{c.4}, one infers that
\begin{equation}\lb{c.5}
\begin{split}
(1-\e)\big|\mathfrak{q}_j(f,f) \big|\leq \e \big| \mathfrak{Q}(f,f)+\mathfrak{q}_1(f,f) \big|+\eta_1(\e)\|f\|_{\cH}^2,&\\
f\in \dom(\mathfrak{Q}),\; 0<\e<\e_1.&
\end{split}
\end{equation}
Now taking $j=2$ in \eqref{c.1} and applying \eqref{c.5}, one obtains
\begin{align}
\big|\mathfrak{q}_2(f,f)\big|&\leq \e \big| \mathfrak{Q}(f,f)+\mathfrak{q}_1(f,f) \big|+\e\big|\mathfrak{q}_1(f,f) \big|+\eta_2(\e)\|f\|_{\cH}^2    \no\\
&\leq \e \big| \mathfrak{Q}(f,f)+\mathfrak{q}_1(f,f) \big|+(1-\e)\big|\mathfrak{q}_1(f,f) \big|+\eta_2(\e)\|f\|_{\cH}^2     \no \\
&\leq 2\e \big| \mathfrak{Q}(f,f)+\mathfrak{q}_1(f,f) \big| + \big[\eta_1(\e)+\eta_2(\e) \big]\|f\|_{\cH}^2,   \lb{c.6} \\
&\hspace*{1.5cm}f\in \dom(\mathfrak{Q}),\; 0<\e<\min\{1/2,\e_1,\e_2\}.\no
\end{align}
Consequently, \eqref{c.2} and \eqref{c.3} follow from \eqref{c.6} by rescaling the parameter $\e$.
\end{proof}

\medskip

{\bf Acknowledgments.} We are indebted to Jonathan Eckhardt, Chris Evans, and Gerald Teschl 
for helpful discussions.   We would also like to thank an anonymous referee for valuable suggestions that improved the presentation of this material.


\end{document}